\documentclass[12pt]{amsart}
\usepackage{amsfonts,amssymb,amsmath,amscd,amsthm, amstext,bm,color}
\usepackage[colorlinks, citecolor=blue,pagebackref,hypertexnames=false]{hyperref}
\usepackage{enumerate}
\usepackage{epstopdf}  
\usepackage{extarrows}
\usepackage{fancybox}
\usepackage{fancyhdr}
\usepackage{geometry}
\usepackage{graphicx}
\usepackage{mathrsfs}
\usepackage{txfonts}
\usepackage{tikz}
\usepackage{ulem}
\usepackage{url}

\allowdisplaybreaks
\newtheorem{theorem}{Theorem}[section]
\newtheorem{proposition}[theorem]{Proposition}
\newtheorem{lemma}[theorem]{Lemma}
\newtheorem{corollary}[theorem]{Corollary}
\newtheorem{remark}[theorem]{Remark}

\title[Maximal regularity in $\rm BMO$]
{Global-in-time maximal regularity for the Cauchy problem \\of the heat equation in $\rm BMO$ and applications}

\author{Xuan Thinh Duong}
\address{Xuan Thinh Duong, School of Mathematical and Physical Sciences, Macquarie University, Sydney, NSW 2019, Australia }
\email{xuan.duong@mq.edu.au}

\author{Ji Li}
\address{Ji Li, School of Mathematical and Physical Sciences, Macquarie University, NSW 2109, Australia}
\email{ji.li@mq.edu.au}

\author{Liangchuan Wu}
\address{Liangchuan Wu,  School of Mathematical Science, Anhui University, Hefei, 230601, P.R.~China 
and 
School of Mathematical and Physical Sciences, Macquarie University, NSW 2109, Australia}
\email{wuliangchuan@ahu.edu.cn}

\author{Lixin Yan}
\address{Lixin Yan, Department of Mathematics,  Sun Yat-sen University,  Guangzhou, 510275,  P.R.~China}
\email{mcsylx@mail.sysu.edu.cn}


\begin{document}

\begin{abstract}
In this article, we establish global-in-time maximal regularity for the Cauchy problem of the classical heat equation $\partial_t u(x,t)-\Delta u(x,t)=f(x,t)$ with $u(x,0)=0$ in a certain $\rm BMO$ setting, which improves the local-in-time result initially proposed by Ogawa and Shimizu in \cite{OS, OS2}.
In further developing our method originally formulated for the heat equation, we obtain analogous global  ${\rm BMO}$-maximal regularity associated to the Schr\"odinger operator  $\mathcal L=-\Delta+V$, where the nonnegative potential $V$ belongs to the reverse H\"older class   ${\rm RH}_q$ for some $q> n/2$.  This extension includes several inhomogeneous estimates as ingredients, such as Carleson-type estimates for the external forces.

Our new methodology is to exploit elaborate heat kernel estimates, along with matched space-time decomposition on the involving integral-type structure of maximal operators, as well as some global techniques such as those from de Simon's work and Schur's lemma. One crucial trick is to utilize the mean oscillation therein to contribute a higher and necessary decay order for global-in-time estimates.
 \end{abstract}

\maketitle

\tableofcontents


\section{Introduction}\label{sec:1}
\setcounter{equation}{0}

The research of maximal regularity for parabolic evolution equations provides an important tool in the theory of nonlinear equations and has attracted lots of attention over the past decades. For example, applying maximal regularity results for the heat equation is a crucial ingredient of the well-posedness for MHD equations in the work \cite{FMRR} by Fefferman, et al. One may refer to the monograph \cite{PS} by Pr\"uss and Simonett, and references therein for this extensive field.
Typically, in the context of the Cauchy problem of the equation $\partial_t u (x,t)+Au(x,t)=f(x,t)$ with $u(x,0)=0$, where $A$ is a  densely closed operator on a Banach space $X$, we say $A$ has the maximal $L^p$-regularity, if for any $f\in L^p(I;X)$, the solution $u$ ensures  that both $\partial_t u$ and $Au$ belongs to $L^p(I;X)$, satisfying
$$
    \|\partial_t u\|_{L^p(I; X)} +\|Au\|_{L^p(I; X)} \leq C\|f\|_{L^p(I; X)},
$$
where $p\in (1,\infty)$ and $I=(0,T)$ for $0<T\leq \infty$. Specifically, the case where $X$ is a Hilbert space traces back to de Simon's work \cite{dS} in 1964, and notably, $X=L^{r}(\mathbb R^n)$ for $1<r<\infty$ represents a typical model.

Notably, when the  Lebesgue spaces $L^r(\mathbb R^n)$ are extended to Banach spaces satisfying the unconditional martingale differences (so-called ``UMD'', and a UMD space is necessarily reflexive), the theory of maximal regularities has been well established, one may refer to  \cite{A1995, A2019, D2021, DHP2003, HNVW, KW2004, W2001, Y2010}  and references therein for instance. One crucial instrument in this topic is the  $H^\infty-$functional calculus, developed by McIntosh and his collaborators (\cite{M86, ADM, CDMY}).

This article pursues the investigation of maximal regularity for the Cauchy problem of the  equation 
\begin{equation}\label{eqn:IVP-classical-heat}
    \begin{cases}
    	\partial_t u(x,t) -\Delta  u(x,t)=f(x,t),\quad (x,t)\in \mathbb R_+^{n+1},\\
    	u(x,0)=0,\quad x\in \mathbb R^n,
    \end{cases}
\end{equation}
in a class of bounded mean oscillation ($\rm BMO$),
where $f=f(x,t):\, \mathbb R^n\times \mathbb R_+\to \mathbb R$ is the  given external term. Recall that   a locally integrable function $g$ is in $\rm BMO(\mathbb R^n)$, the space of bounded mean oscillation, if 
$$
    \|g\|_{{\rm BMO}(\mathbb R^n)}:=\sup_{B=B(x_B,r_B)\subset \mathbb R^n}\left(r_B^{-n}\int_B \left|g(x)-g_B\right|^2  dx\right)^{1\over 2}<\infty,
$$ 
where $g_B$ denotes the average of $g$ over the ball $B$, i.e., $g_B=|B|^{-1} \int_B g(x) dx$. Observe that ${\rm BMO}$ is not UMD due to  its non-reflexivity:
$$
    \big({\rm CMO}(\mathbb R^n)\big)^{*}= H^1(\mathbb R^n)\quad \text{and}  \quad  (H^1 \mathbb R^n))^*={\rm BMO}(\mathbb R^n),
$$ 
see \cite{CW} by Coifman and Weiss,  and \cite{FS} by Fefferman and Stein, respectively. Here $H^1(\mathbb R^n)$ is the classical Hardy space, and ${\rm CMO}(\mathbb R^n)$, the space of vanishing mean oscillation,  using Uchiyama's notation in \cite{U}, defined as the closure of $C_c^{\infty}(\mathbb R^n)$ in the ${\rm BMO}(\mathbb R^n)$ norm, is a proper subspace of ${\rm BMO}(\mathbb R^n)$. Therefore maximal regularity in the $\rm BMO$ setting does not match the UMD framework and needs some new techniques.

Recently, Ogawa and Shimizu in \cite{OS2} (and their previous works \cite{OS}) established   local-in-time maximal regularity  in $\rm BMO$  for \eqref{eqn:IVP-classical-heat} that  
\begin{equation}\label{result-OS}
    \|\partial_t u\|_{\widetilde{L^2}(\mathbb R_+; {\rm BMO}(\mathbb R^n))} \leq C  \|f\|_{\widetilde{L^2}(\mathbb R_+; {\rm BMO}(\mathbb R^n))},
\end{equation}
where $\widetilde{L^2}(\mathbb R_+; {\rm BMO}(\mathbb R^n))$ is a substitution  for the Bochner class $L^2(\mathbb R_+; {\rm BMO}(\mathbb R^n))$, with the quasi-norm 
$$
      \|f\|_{\widetilde{L^2}(\mathbb R_+; {\rm BMO}(\mathbb R^n))}^2 :=\sup_{B=B(x_B, r_B)\subset \mathbb R^n}  \int_{0}^{r_B^2} \frac{1}{|B|^2}\int_{B\times B} \big|f(x,t)-f(y,t) \big|^2 dxdydt<\infty,
$$
where the time interval is constrained to the scale of $r_B^2$ for each specific ball $B$. This quasi-norm also has an equivalent expression
\begin{equation}\label{eqn:L2BMO}
   \|f\|_{\widetilde{L^2}(\mathbb R_+; {\rm BMO}(\mathbb R^n))}^2
  \approx  \sup_{B=B(x_B, r_B)\subset \mathbb R^n} r_B^{-n} \int_{0}^{r_B^2} \int_{B} \left|f(x,t)- \left(f(\cdot, t)\right)_{B} \right|^2 dxdt,
\end{equation}
as detailed in the beginning of Section \ref{sec:clasBMO}. One will see that this revised definition \eqref{eqn:L2BMO} and its global extension \eqref{def:L2BMO-new} below are more practical in our argument framework.

The above maximal regularity \eqref{result-OS} was derived  from
\begin{equation}\label{eqn:IVP-heat-I-classical}
    \big\| \mathcal M_{\Delta} (f)\big\|_{\widetilde{L^2}(\mathbb R_+;{\rm BMO})}\leq C \|f\|_{\widetilde{L^2}(\mathbb R_+;{\rm BMO})}
\end{equation}
for {\it maximal regularity operator} 
\begin{equation}\label{MRO}
    \mathcal M_{\Delta} (f)(x,t):=\int_0^t \partial_t e^{(t-s)\Delta}f(\cdot,s)(x)\,ds.
\end{equation}
Moreover, they also showed that 
$$
    \big\| \mathcal T_{\Delta} f\big\|_{\widetilde{L^2}(\mathbb R_+;{\rm BMO})}\leq C \|f\|_{\widetilde{L^1}(\mathbb R_+;{\rm BMO})}
$$
for the analogous regularity operator 
\begin{equation}\label{MRO2}
    \mathcal T_{\Delta} (f)(x,t):=\int_0^t \nabla_x e^{(t-s)\Delta}f(\cdot,s)(x)\,ds,
\end{equation}
where $\|f\|_{\widetilde{L^1}(\mathbb R_+;{\rm BMO})}$ is defined in a similar manner. By their approach, one may obtain a variant  that
\begin{equation}\label{eqn:IVP-heat-II-variant-classical}
    \left\| t^{-1/2} \mathcal T_{\Delta} f\right\|_{\widetilde{L^2}(\mathbb R_+;{\rm BMO})}\leq C \|f\|_{\widetilde{L^2}(\mathbb R_+;{\rm BMO})}.
\end{equation}

The method in \cite{OS2}  is intrinsically based on integration by parts which leads to estimates on $\nabla_x u$. 
It also applies the commutative property between the derivative $\nabla_x$ and the heat semigroup $e^{t\Delta}$.

Remarkably, it was noted in \cite[Section 1]{OS2} that it's impossible to establish maximal regularity for the heat equation \eqref{eqn:IVP-classical-heat} in $L^2((0,T);{\rm BMO})$ for $0<t\leq \infty$, from the perspective of the failure for the homogeneous estimate in $\rm BMO$. Alternatively, it also follows from the fact that the operator $-\Delta$ does not have maximal $L^p-$regularity on any space which contains $C_0(\mathbb R^n)$ isomorphically for any $p\in [1,+\infty]$; see \cite[p. 660]{HNVW}.

In this article, we will improve the local maximal-regularity \eqref{result-OS}  and the corresponding estimates \eqref{eqn:IVP-heat-I-classical} and \eqref{eqn:IVP-heat-II-variant-classical} to their global-in-time versions, respectively. Concretely,  we say a measurable function $f(x,t)$ is in ${\mathbb L}^2(\mathbb R_+;{\rm BMO}(\mathbb R^n))$, if 
\begin{equation}\label{def:L2BMO-new}
   \|f\|_{{\mathbb L}^2(\mathbb R_+;{\rm BMO}(\mathbb R^n))}^2  =   \sup_{B=B(x_B,r_B)\subset \mathbb R^n}   r_B^{-n}\int_{\mathbb R_+}  \int_B \big|f(x,t)- \left(f(\cdot, t)\right)_B\big|^2 dx dt  <\infty.
\end{equation}
Note that 
$$
 L^2(\mathbb R_+;{\rm BMO})\subset {\mathbb L}^2(\mathbb R_+;{\rm BMO}) \subset  \widetilde{L^2}(\mathbb R_+;{\rm BMO}),
$$
and the only slight difference between the (semi-)norms of ${\mathbb L}^2(\mathbb R_+;{\rm BMO})$ and $L^2(\mathbb R_+;{\rm BMO})$   lies in the order of taking the 
supremum and the time integration. Due to the failure of maximal regularity in $L^2(\mathbb R_+;{\rm BMO})$, our result about maximal regularity in   ${\mathbb L}^2(\mathbb R_+;{\rm BMO})$ (i.e., Theorem \ref{thm:endpoint-HE} below) is rather satisfactory.
Similarly, we can define $ \dot{\mathbb W}^{1,2}(\mathbb R_+; {\rm BMO}(\mathbb R^n))$ for $\partial_t f\in  {\mathbb L^2}(\mathbb R_+;{\rm BMO}(\mathbb R^n))$, and define $  {\mathbb L^2}(\mathbb R_+; \dot{{\rm BMO}}^2(\mathbb R^n))$ for $|\nabla_x|^2  f\in{\mathbb L}^2(\mathbb R_+;{\rm BMO}(\mathbb R^n))$, respectively, where $\dot{\rm BMO}^2(\mathbb R^n)=\left\{g \text{ is\  in\  tempered \ distribution}:  \   |\nabla_x|^2 g\in {\rm BMO} (\mathbb R^n)\right\}$. 

Our first main result is the following global maximal regularity for the Cauchy problem of the heat equation \eqref{eqn:IVP-classical-heat} in a certain $\rm BMO$ setting. 

\begin{theorem}\label{thm:endpoint-HE}   
For every $f\in {\mathbb L}^2 (\mathbb R_+;{\rm BMO}(\mathbb R^n))$,  we have
\begin{equation}\label{eqn:MaxReg-II-classical}
	    \big\|\mathcal M_{\Delta} f \big\|_{{\mathbb L}^2(\mathbb R_+; {\rm BMO})}\leq C \|f\|_{{\mathbb L}^2(\mathbb R_+; {\rm BMO})}
\end{equation} 
and
\begin{equation}\label{eqn:MaxReg-I-classical}
    \left\|t^{-1/2}\mathcal T_{\Delta} f \right\|_{{\mathbb L}^2(\mathbb R_+; {\rm BMO})}\leq C \|f\|_{{\mathbb L}^2(\mathbb R_+; {\rm BMO})}.
\end{equation}

As a consequence, for every $f\in  {\mathbb L}^2(\mathbb R_+;{\rm BMO}(\mathbb R^n))$,  the Cauchy problem of the heat equation \eqref{eqn:IVP-classical-heat}  admits a unique (up to constant) solution $u\in  \dot{\mathbb W}^{1,2}  (\mathbb R_+; {\rm BMO}(\mathbb R^n))\cap   {\mathbb  L^2}(\mathbb R_+; \dot{{\rm BMO}}^2(\mathbb R^n))$ which satisfies
\begin{equation}\label{eqn:MaxReg-III-classical}
       \left\|t^{-1/2}\partial_x u\right\|_{{\mathbb L}^2(\mathbb R_+; {\rm BMO})}+ 
       \left\|\partial_t u\right\|_{{\mathbb L}^2(\mathbb R_+; {\rm BMO})} \leq C\|f\|_{{\mathbb L}^2(\mathbb R_+; {\rm BMO})}.
\end{equation}
\end{theorem}

Compared with $\widetilde{L^2}(\mathbb R_+; {\rm BMO})$ in \eqref{eqn:L2BMO}, our  ${\mathbb L}^2(\mathbb R_+; {\rm BMO})$ in \eqref{def:L2BMO-new} extends the integral interval from a local scale to the entire $\mathbb R_+$. Thus our Theorem \ref{thm:endpoint-HE}  improves the work of Ogawa and Shimizu on maximal regularity for the heat equation to a global result. Moreover,  even at the local level, our approach of Theroem \ref{thm:endpoint-HE} can refine and improve their maximal regularity estimates \eqref{eqn:IVP-heat-I-classical} and \eqref{eqn:IVP-heat-II-variant-classical} into corresponding Carleson-type estimates, respectively; see Corollary \ref{coro:compare-1} for details.

Contrasting with integration by parts in \cite{OS2}, our argument for Theorem \ref{thm:endpoint-HE}  can be regarded as a certain local-global analysis. This primarily focuses on exploiting elaborate heat kernel estimates, along with matched space-time decomposition on the involving integral-type structure of maximal operators. Meanwhile, we apply some global techniques such as those from de Simon's work \cite{dS} on the $L^2(\mathbb R_+^{n+1})$ boundedness of $\mathcal M_{\Delta}$ and Schur's lemma, when the operators therein are resistant to deformation processing or cannot be localized effectively. Among these, one key observation is that the mean oscillation in the definition \eqref{def:L2BMO-new} of ${\mathbb L}^2(\mathbb R_+; {\rm BMO})$ can deduce an additional derivative via Poincar\'e's inequality, which will contribute a higher and necessary decay order for global-in-time estimates.
Notably,  we bypass the commutation between $\nabla_x$ and the heat semigroup. All of these ensure a broader applicability of our methodology to more general differential operators and related topics.

As an application, in the second part of this article, we are devoted to establishing an analog of global maximal regularity associated to Schr\"odinger operators. 
Let
\begin{equation*} 
    \mathcal{L}=-\Delta+V(x) \   \    { \rm on } \  \  L^{2}(\mathbb{R}^{n}), \quad n \geq 3,
\end{equation*}
where the nonnegative potential $V$ is not identically zero, and belongs to the reverse H\"older class ${\rm RH}_q$  for  some $q>n/2$, in the sense of 
\begin{equation}\label{eqn:Reverse-Holder}
    \left(\frac{1}{|B|} \int_{B} V(y)^{q} d y\right)^{1 / q} \leq \frac{C}{|B|} \int_{B} V(y)\, d y.
\end{equation}

Consider
 \begin{equation}\label{eqn:IVP-heat}
    \begin{cases}
    	\partial_t u +\mathcal L u=f,\quad (x,t)\in \mathbb R_+^{n+1},\\
    	u(\cdot,0)=0,\quad x\in \mathbb R^n.
    \end{cases}
\end{equation}
Since $-\mathcal L$ generates an analytic semigroup $\left\{e^{-t\mathcal L}\right\}_{t\geq 0}$, the solution of \eqref{eqn:IVP-heat} is given by the Duhamel formula
$$
    u(x, t)=\int_0^t e^{-(t-s)\mathcal L}f(\cdot,s)(x)\,ds,\quad t> 0.
$$
Note that the kernel of the heat semigroup $e^{-t\mathcal L}$ generated by $-\mathcal L$ satisfies Gaussian upper estimates (see Lemma~\ref{lem:heat-Schrodinger} below), consequently the work \cite{CD2000} by Coulhon and the first author of this article, within an abstract framework, yields that $\mathcal L$ has the maximal $L^p$ regularity property in $L^r(\mathbb R^n)$ for $1<p,r<\infty$. That is,
$$
   \|\partial_t u\|_{L^p(\mathbb R_+;L^r(\mathbb R^n))}   \leq C \|f\|_{L^p(\mathbb R_+; L^r(\mathbb R^n))}.
$$
It's natural to consider maximal regularity for \eqref{eqn:IVP-heat} in  a class of $\rm BMO$ associated to $\mathcal L$, as previously done for the classical heat equation.

We begin by introducing some necessary notation.
Recall that  $g$ belongs to  ${\rm BMO}_{{\mathcal{L}}}({\mathbb R}^n)$ (see \cite{DGMTZ}) if  $g$ is a locally integrable function and satisfies
\begin{align*}
     \|g\|_{{\rm BMO}_{\mathcal L}(\mathbb{R}^n)}^2  
    := &\max\left\{\sup_{B=B(x_B,r_B):\, r_B<\rho(x_B)} r_B^{-n} \int_{B}\left|g(y)-g_{B}\right|^2   d y, \right.  \\
    &\qquad \qquad   \left. \sup_{B=B(x_B,r_B):\, r_B\geq \rho(x_B)}  r_B^{-n}  \int_{B}|g(y)|^2\, d y\right\} <\infty.
\end{align*}
Here, the function $\rho(x)$,  introduced by Z.W. Shen \cite{Shen,Shen1},  is defined by
\begin{equation}\label{eqn:critical-funct}
 \rho(x)=\sup \left\{r>0: \frac{1}{r^{n-2}} \int_{B(x, r)} V(y)\, d y \leq 1\right\}.
\end{equation}
Note that this ${\rm BMO}_{\mathcal L}(\mathbb{R}^n)$ space is a proper subspace of the classical $\rm BMO$ space and when $V\equiv 1$, ${\rm BMO}_{-\Delta+1}$ is  the ${\rm bmo}$ space  introduced by Goldberg \cite{Go1}. For further characterizations and properties of ${\rm BMO}_{\mathcal L}$, one may refer to \cite{DY1, DY2} for details.

Similarly, the space ${\rm BMO}_{\mathcal L}$ does not  fit within the ``UMD'' frame since
 $$
     \left({\rm CMO}_{\mathcal L}(\mathbb R^n)\right)^{**}={\rm BMO}_{\mathcal L}(\mathbb R^n),
 $$
 where ${\rm CMO}_{\mathcal{L}}(\mathbb R^n)$ defined as the closure of $C_c^{\infty}(\mathbb R^n)$ in the ${\rm BMO}_{\mathcal{L}}(\mathbb R^n)$ norm is a proper subspace of ${\rm BMO}_{\mathcal L}(\mathbb R^n)$; see \cite{SW2022} by Song and the third author for more details. Fortunately, our argument for Theorem~\ref{thm:endpoint-HE} is applicable here.
 
We say a  function $f(x,t)$ belongs to ${\mathbb L}^2(\mathbb R_+;{\rm BMO}_{\mathcal L}(\mathbb R^n))$, if
\begin{align*}
  \|f\|_{{\mathbb L}^2(\mathbb R_+;{\rm BMO}_{\mathcal L})} ^2   :=&  \max\left\{ \sup_{B=B(x_B,r_B):\, r_B<\rho(x_B)}   r_B^{-n} \int_{\mathbb R_+} \int_B \left|f(x,t)- \left(f(\cdot, t)\right)_B \right|^2 dxdt,  \right.\nonumber\\ 
    & \qquad\quad   \left.  \sup_{B=B(x_B,r_B):\, r_B\geq \rho(x_B)} r_B^{-n}  \int_{\mathbb R_+}  \int_B |f(x,t)|^2dxdt\right\}
     <\infty.
\end{align*}
Note that $\|\cdot\|_{{\mathbb L}^2(\mathbb R_+;{\rm BMO}_{\mathcal L}) }$ is indeed a norm (which is not only a seminorm) and also
$$
 L^2(\mathbb R_+;{\rm BMO}_{\mathcal L}(\mathbb R^n))\subset {\mathbb L}^2(\mathbb R_+;{\rm BMO}_{\mathcal L} (\mathbb R^n)) \subset  \widetilde{L^2}(\mathbb R_+;{\rm BMO}_{\mathcal L} (\mathbb R^n)).
$$
Similarly, we can define $ \dot{\mathbb W}^{1,2}   (\mathbb R_+; {\rm BMO}_{\mathcal L}(\mathbb R^n))$ for $\partial_t f\in {\mathbb L}^2(\mathbb R_+;{\rm BMO}_{\mathcal L}(\mathbb R^n))$, and define ${\mathbb L}^2(\mathbb R_+; \dot{{\rm BMO}}_{\mathcal L}^2(\mathbb R^n))$ for $\mathcal L f=(\sqrt{\mathcal L})^2 f\in {\mathbb L}^2(\mathbb R_+;{\rm BMO}_{\mathcal L}(\mathbb R^n))$, respectively.

Let $\mathcal M_{\mathcal L}$ be the maximal regularity operator associated to $\mathcal L$:
$$
    \mathcal M_{\mathcal L} (f)(x,t)=\int_0^t \partial_t e^{-(t-s)\mathcal L}f(\cdot,s)(x)\,ds.
$$
that is, the operator in \eqref{MRO} replacing $\Delta$ therein by  $-\mathcal L$.

The following is our second main result. 

\begin{theorem}\label{thm:improved-mainA}   
Suppose $V \in \mathrm{RH}_q$ for some $q>n / 2$. For every $f\in  {\mathbb L}^2(\mathbb R_+;{\rm BMO}_{\mathcal L}(\mathbb R^n))$,  we have
\begin{equation}\label{eqn:global-MaxReg-II}
	    \big\|\mathcal M_{\mathcal L} f \big\|_{{\mathbb L}^2(\mathbb R_+; {\rm BMO}_{\mathcal L})}\leq C \|f\|_{{\mathbb L}^2(\mathbb R_+; {\rm BMO}_{\mathcal L})}.
\end{equation} 

As a consequence, for every $f\in {\mathbb L}^2(\mathbb R_+;{\rm BMO}_{\mathcal L}(\mathbb R^n))$,  the Cauchy problem of the Schr\"odinger equation \eqref{eqn:IVP-heat} admits a unique solution $u\in    \dot{\mathbb W}^{1,2} (\mathbb R_+; {\rm BMO}_{\mathcal L} (\mathbb R^n))\cap {\mathbb L}^2(\mathbb R_+; \dot{{\rm BMO}}_{\mathcal L}^2(\mathbb R^n))$ which satisfies
\begin{equation}\label{eqn:global-MaxReg-III}
    \left \|\partial_ t u\right \|_{{\mathbb L}^2(\mathbb R_+;{\rm BMO}_{\mathcal L})} \leq C\|f\|_{{\mathbb L}^2(\mathbb R_+;{\rm BMO}_{\mathcal L})}.
\end{equation}
\end{theorem}

Note that the conservation law   $e^{t\Delta} 1=1$ simplifies the argument of Theorem \ref{thm:endpoint-HE}. This does not hold for the heat semigroup $e^{-t\mathcal L}$. To prove Theorem \ref{thm:improved-mainA}, we will still apply the argument for showing Theorem \ref{thm:endpoint-HE} in a more subtle manner, combined with the upper bound estimates on heat kernels of $e^{-t\mathcal L}$ and their derivatives, and the slowly varying property of  $\rho$.

It's interesting to consider the global estimate for $\mathcal T_{\mathcal L}$, the regularity operator in \eqref{MRO2} with $\Delta$ therein replaced by $-\mathcal L$, the difficulty arises from 
the lack of regularity on heat kernels, one may refer to  Remark \ref{rem:global-TL} for further discussion.

The layout of the article is as follows. 
Our purpose in Section \ref{sec:clasBMO} is to prove Theorem \ref{thm:endpoint-HE}, the global-in-time maximal regularity for the Cauchy problem of the classical heat equation in $\rm BMO$.
As an application of Theorem \ref{thm:endpoint-HE}, in the next we study this topic associated to Schr\"odinger operators $\mathcal L=-\Delta+V$. To this end, we study local-in-time  maximal regularity associated to ${\rm BMO}_{\mathcal L}$ in Section \ref{sec:local-BMOL} as ingredients. Eventually, in Section \ref{sec:global-BMOL}, we prove Theorem \ref{thm:improved-mainA}, the second main result of our article.

\medskip


\section{Global-in-time  maximal regularity in ${\rm BMO}$ }\label{sec:clasBMO}
\setcounter{equation}{0}

In this section, we aim to prove Theorem \ref{thm:endpoint-HE}.

To avoid confusion, we first clarify that the original definition of the quasi-norm of $\widetilde{L^2}(\mathbb R_+; {\rm BMO})$ in \cite{OS2}, is indeed equivalent to that in \eqref{eqn:L2BMO}. It suffices to show
\begin{align}\label{equiv-L2BMO}
    \int_{0}^{r_B^2} \frac{1}{|B|^2}\int_{B\times B} \big|f(x,t)- f(y,t) \big|^2 dxdydt 
    \approx     \int_{0}^{r_B^2} \frac{1}{|B|}\int_{B} \big|f(x,t)- \left(f(\cdot, t)\right)_{B} \big|^2 dxdt,
\end{align}
with the implicit constant independent of the ball $B$. 
Note that
\begin{align*}
    \text{LHS \  of\ } \eqref{equiv-L2BMO} &= \int_{0}^{r_B^2} \frac{1}{|B|^2}\int_{B\times B} \big| \left(f(x,t)-(f(\cdot, t))_B\right) -   \left(f(y,t)-(f(\cdot, t))_B\right) \big|^2 dxdydt \\
    &\leq 2\int_{0}^{r_B^2} \frac{1}{|B|^2}\int_{B\times B} \big| f(x,t)-(f(\cdot, t))_B \big|^2 dxdydt + 2\int_{0}^{r_B^2} \frac{1}{|B|^2}\int_{B\times B} \big|   f(y,t)-(f(\cdot, t))_B \big|^2 dxdydt \\
    &= 4 \cdot  \text{RHS \  of\ } \eqref{equiv-L2BMO}   ,
\end{align*}
and it follows from H\"older's inequality that 
\begin{align*}
    \text{RHS \  of\ } \eqref{equiv-L2BMO} &= \int_{0}^{r_B^2} \frac{1}{|B|}\int_{B} \bigg|   \frac{1}{|B|} \int_B\left(f(x,t)- f(y,t)\right) dy \bigg|^2 dxdt\\
    &\leq   \text{LHS \  of\ } \eqref{equiv-L2BMO} ,  
\end{align*}
as desired.  

This fact, combined with the definition  \eqref{def:L2BMO-new} of the global-in-time ${\mathbb L}^2(\mathbb R_+;{\rm BMO}(\mathbb R^n))$ quasi-norm, ensures clearly that our Theorem \ref{thm:endpoint-HE} improve the local-in-time result on maximal regularity for the Cauchy problem of the heat equation in \cite{OS2} to a global-in-time version.

Recall that $\mathcal M_{\Delta} (f)$ and $\mathcal T_{\Delta}  (f)$ are  the regularity operators given in \eqref{MRO} and \eqref{MRO2}, respectively.

\begin{proof}[Proof of Theorem~\ref{thm:endpoint-HE}]
For any fixed $B=B(x_B,r_B)$, we  split  
\begin{align*}
    f(\cdot, s)&=\left[f(\cdot, s)-\left(f(\cdot, s)\right)_{4 B}\right] \mathsf 1_{4 B}+\left[f(\cdot,s)-\left(f(\cdot, s)\right)_{4 B}\right] \mathsf 1_{(4 B)^c}+\left(f(\cdot, s)\right)_{4 B}\\
   & =: f_1(\cdot, s)+f_2(\cdot, s)+f_3(\cdot, s),
\end{align*}

\noindent{\it Step I.} We begin by proving \eqref{eqn:MaxReg-I-classical}, which seems easier to estimate by applying  H\"older's inequality directly. 

Let 
$$
    \mathbf I_{B,j} :=  r_B^{-n}\int_{\mathbb R_+}   \int_B \left| t^{-1/2}\mathcal T_{\Delta} (f_j)(x,t)- \left(t^{-1/2}\mathcal T_{\Delta} (f_j)(\cdot, t)\right)_B\right|^2  dx dt ,\quad j=1,2,3.
$$

Observe that by H\"older's inequality,
\begin{align*}
   \mathbf I_{B,1} &\lesssim r_B^{-n}   \int_{\mathbb R_+}   \int_B \bigg|  \int_0^t t^{-1/2}\nabla_x e^{(t-s)\Delta} f_1 (\cdot, s)(x)\,ds\bigg|^2  dx dt\\
   &\lesssim
    r_B^{-n}  \int_{\mathbb R_+}  \int_{\mathbb R_+} \int_{B} \left|\nabla_x (-\Delta)^{-1/2} \sqrt{-\Delta}\, e^{\tau \Delta} f_1(\cdot, s)(x)\right|^2 dxd\tau ds.
\end{align*}

One may apply the $L^2$-boundedness of the Riesz transform and spectral theory \cite{Yo}  ( see also \eqref{eqn:FC} which holds when replacing $\mathcal L$ by $-\Delta$) to see
\begin{align}\label{eqn:classical-IB1}
 \mathbf I_{B,1}&\lesssim    r_B^{-n}  \int_{\mathbb R_+}   \left[\int_{\mathbb R_+} \int_{\mathbb R^n} \left|  \sqrt{-\tau\Delta} e^{\tau\Delta} f_1 (\cdot, s)(x)\right|^2  \frac{dxd\tau}{\tau} \right] ds \nonumber\\
 &\lesssim r_B^{-n}  \int_{\mathbb R_+}     \int_{4B}\left|f(x,s)-\left(f(\cdot,s)\right)_{4B}\right|^2 dx   ds \\
	&\lesssim\|f\|_{{\mathbb L}^2(\mathbb R_+;{\rm BMO})}^2.  \nonumber
\end{align}
One may note that at the local scale of time interval, the mean oscillation (i.e., subtracting the integral average on a ball) is not necessary, and the above estimate can be viewed as some Carleson-type estimates for the external term $f_1(x,t)$. In a local sense, the estimate on $\mathbf I_{B,1}$ is  more precise than what we initially sought. See Corollary \ref{coro:compare-1} for further discussion.
 
Note that  the conservation law $e^{\tau\Delta }(1)=1$ for any $\tau>0$ yields the cancellation
$$
    \nabla_x e^{(t-s) \Delta } f_3(\cdot, s)\equiv 0,\quad \forall\   0<s<t,
$$
hence 
$$
    \mathbf I_{B,3}=0.
$$

It suffices to estimate $ \mathbf I_{B,2}$ which is more elaborate. By Poincar\'e's inequality for balls (see for example \cite[p. 291]{E}), there exists a constant $C$ depending only on $n$, such that 
\begin{align*}
     \mathbf I_{B,2} &\leq C\int_{\mathbb R_+}  \sum_{i,j=1}^n r_B^{-n+2} \int_B \left|t^{-1/2}  \nabla_x \mathcal T_{\Delta}  (f_2)(x,t)\right|^2 dx dt\\
     &\lesssim    \sum_{i,j=1}^n   \int_{\mathbb R_+}   r_B^{-n+2} \int_B  \int_0^t  \bigg|   \frac{\partial ^2e^{(t-s)\Delta } f_2(\cdot, s)(x)}{\partial x_i \partial x_j}   \bigg|^2 ds dxdt  \\
     &\lesssim     \sum_{i,j=1}^n \int_{\mathbb R_+}  r_B^{-n+2}\int_{\mathbb R_+}  \int_B  \bigg|    \frac{\partial ^2e^{\tau\Delta } f_2(\cdot, s)(x)}{\partial x_i \partial x_j} \bigg|^2 dxd\tau   ds\\
    &\lesssim  \int_{\mathbb R_+} r_B^{-n+2}     \int_{0}^{(4r_B)^2 }  \int_B  \left|\int_{(4B)^c} \frac{1}{|x-y|^{n+2}}   \left|f(y,s)-\left(f(\cdot, s)\right)_{4 B}\right| dy\right|^2 dx d\tau ds\\
    &\quad + \sum_{k=1}^\infty    \int_{\mathbb R_+} r_B^{-n+2}   \int_{(4^kr_B)^2}^{(4^{k+1}r_B)^2} \int_B   \left| \int_{(4B)^c}    \min \left\{ \frac{1}{\tau^{\frac{n+2}{2}} },  \frac{1}{|x-y|^{n+2}}  \right\}   \left|f(y,s)-\left(f(\cdot, s)\right)_{4 B}\right| dy  \right|^2  dxd\tau ds\\
    &=:    \mathbf I_{B,2}^{(0)}  + \sum_{k=1}^\infty    \mathbf  I_{B,2}^{(k)} .
\end{align*}

Note that 
$$
    \left|f(y,s)-\left(f(\cdot, s)\right)_{4B}\right|  \leq  \left|f(y,s)-\left(f(\cdot, s)\right)_{4^{l+1}B}\right| 
    +\left|\left(f(\cdot,s)\right)_{4^{l+1}B}-\left(f(\cdot, s)\right)_{4B}\right|
$$
for any $l\geq 1$, and 
\begin{align*}
    \left|\left(f(\cdot,s)\right)_{4^{l+1}B}-\left(f(\cdot, s)\right)_{4B}\right| &\leq \sum_{m=1}^l 	\left|\left(f(\cdot,s)\right)_{4^{m+1}B}-\left(f(\cdot, s)\right)_{4^m B}\right| \\
    &\leq 4^n  \sum_{m=1}^l  \frac{1}{|4^{m+1}B|} \int_{4^{m+1}B} \left|f(z,s)-\left(f(\cdot, s)\right)_{4^{m+1}B}\right| dz,
\end{align*}
we have 
\begin{align*}
    &\int_{(4B)^c} \frac{1}{|x-y|^{n+2}}   \left|f(y,s)-\left(f(\cdot, s)\right)_{4 B}\right| dy \\
    \lesssim \,& \sum_{l=1}^\infty    \int_{4^{l+1}B\setminus {4^l B}}     \frac{1}{(4^l r_B)^{n+2}}   \left|f(y,s)-\left(f(\cdot, s)\right)_{4^{l+1}B}\right|  dy   \\
    & +    \sum_{l=1}^\infty    \int_{4^{l+1}B\setminus {4^l B}}     \frac{1}{(4^l r_B)^{n+2}}        \sum_{m=1}^l  \frac{1}{|4^{m+1}B|} \int_{4^{m+1}B} \left|f(z,s)-\left(f(\cdot, s)\right)_{4^{m+1}B}\right| dzdy\\
    \lesssim\,  & \sum_{l=1}^\infty    \frac{1}{(4^l r_B)^2}        \sum_{m=1}^l  \frac{1}{|4^{m+1}B|} \int_{4^{m+1}B} \left|f(z,s)-\left(f(\cdot, s)\right)_{4^{m+1}B}\right| dz\\
\lesssim\, & \frac{1}{r_B^2} \sum_{m=1}^\infty \sum_{l=m}^\infty \frac{1}{4^{2l}}  \frac{1}{|4^{m+1}B|}\int_{4^{m+1}B} \big|f(z,s)-(f(\cdot,s))_{4^{m+1}B} \big| \, dz \\
\lesssim \,&  \frac{1}{r_B^2}      \sum_{m=1}^\infty  \frac{1}{4^{2m}}   \frac{1}{|4^{m+1}B|}\int_{4^{m+1}B} \big|f(z,s)-(f(\cdot,s))_{4^{m+1}B} \big| \,dz.
\end{align*}
By Cauchy-Schwarz inequality and H\"older's inequality, 
\begin{align}\label{eqn:classical-IB20}
     \mathbf I_{B,2}^{(0)} &\lesssim \int_{\mathbb R_+}  r_B^2 \int_0^{(4r_B)^2}  \frac{1}{r_B^4} d\tau    \sum_{m=1}^\infty \frac{1}{4^{2m}}   \cdot  \sum_{m=1}^\infty \frac{1}{4^{2m}}       \bigg(\frac{1}{|4^{m+1}B|}\int_{4^{m+1}B} \big|f(z,s)-(f(\cdot,s))_{4^{m+1}B} \big| \,dz\bigg)^2   ds \nonumber\\
      &\lesssim \int_{\mathbb R_+}  r_B^2 \int_0^{(4r_B)^2}  \frac{1}{r_B^4} d\tau \sum_{m=1}^\infty \frac{1}{4^{2m}}    \frac{1}{|4^{m+1}B|}\int_{4^{m+1}B} \big|f(z,s)-(f(\cdot,s))_{4^{m+1}B} \big|^2 dzds \nonumber\\
     &\lesssim   \sum_{m=1}^\infty   \frac{1}{4^{2m}}   \int_{\mathbb R_+}    \frac{1}{|4^{m+1}B|}\int_{4^{m+1}B} \big|f(z,s)-(f(\cdot,s))_{4^{m+1}B} \big|^2 dzds \\
     &\lesssim \|f\|_{{\mathbb L}^2(\mathbb R_+;{\rm BMO})}^2. \nonumber
\end{align}

Similarly, for any  $\tau\in \left((4^k r_B)^2, (4^{k+1}r_B)^2\right)$ with $k\geq 1$,
\begin{align*}
& \int_{(4B)^c}    \min \left\{ \frac{1}{\tau^{\frac{n+2}{2}} },  \frac{1}{|x-y|^{n+2}}  \right\}   \left[f(y,s)-\left(f(\cdot, s)\right)_{4 B}\right] dy \\
\lesssim\, &  \frac{1}{\tau^{\frac{3}{4}}} \sum_{l=1}^\infty \int_{4^{l+1}B\setminus {4^l B}}    \frac{1}{(4^l r_B)^{n+\frac{1}{2}}}\sum_{m=1}^l \frac{1}{|4^{m+1}B|}\int_{4^{m+1}B} \big|f(z,s)-(f(\cdot,s))_{4^{m+1}B} \big|\,dz dy\\
\lesssim\,& \frac{1}{r_B^2} \frac{1}{4^{\frac{3k}{2}}}     \sum_{m=1}^\infty  \frac{1}{2^m}   \frac{1}{|4^{m+1}B|}\int_{4^{m+1}B} \big|f(z,s)-(f(\cdot,s))_{4^{m+1}B} \big|\,dz,
\end{align*}
hence for $k\geq 1$,
\begin{align}\label{eqn:classical-IB2k}
 \mathbf I_{B,2}^{(k)} &\lesssim \int_{\mathbb R_+} r_B^{2} \int_{(4^kr_B)^2}^{(4^{k+1}r_B)^2}    \frac{1}{r_B^4} d\tau\,\frac{1}{4^{3k}}   \sum_{m=1}^\infty \frac{1}{2^{m}}    \frac{1}{|4^{m+1}B|}\int_{4^{m+1}B} \big|f(z,s)-(f(\cdot,s))_{4^{m+1}B} \big|^2 dzds\nonumber\\
  \lesssim \, & \frac{1}{4^k}   \sum_{m=1}^\infty  \frac{1}{2^{m}}     \int_{\mathbb R_+}  \frac{1}{|4^{m+1}B|}\int_{4^{m+1}B} \big|f(z,s)-(f(\cdot,s))_{4^{m+1}B} \big|^2 dzds\\
 \lesssim \, & \frac{1}{4^k}  \|f\|_{{\mathbb L}^2(\mathbb R_+;{\rm BMO})}^2. \nonumber
\end{align}

Thus,
$$
 \mathbf I_{B,2} \lesssim \mathbf I_{B,2}^{(0)}  + \sum_{k=1}^\infty    \mathbf  I_{B,2}^{(k)} \lesssim \|f\|_{{\mathbb L}^2(\mathbb R_+;{\rm BMO})}^2,
$$
that is, we obtain \eqref{eqn:MaxReg-I-classical}, as desired.

\smallskip

\noindent{\it Step II.} It remains to show \eqref{eqn:MaxReg-II-classical}.
Let
$$
    \mathbf J_{B,j} :=  r_B^{-n}\int_{\mathbb R_+}  \int_B \left|\mathcal M_{\Delta} (f_j)(x,t)- \left(\mathcal M_{\Delta} (f_j)(\cdot, t)\right)_B\right|^2 dx dt ,\quad j=1,2,3.
$$

Similarly, the conservation law $e^{\tau \Delta }(1)=1$ for $\tau>0$ yields  $ \Delta e^{(t-s) \Delta } f_3(\cdot, s)\equiv 0$ for any  $s<t$, and this implies that 
$$
    \mathbf J_{B,3}=0.
$$
Hence it suffices to estimate $\mathbf J_{B,1}$ and $\mathbf J_{B,2}$.

Since the Laplace operator $\Delta$ generates a bounded holomorphic semigroup in  $L^2(\mathbb R^n)$, it follows from de Simon's work \cite{dS}  to see the maximal operator  $\mathcal M_{\Delta} $  is a bounded operator on $L^2(\mathbb R_+^{n+1})$. Therefore,
\begin{align}\label{eqn:classical-JB1}
    \mathbf J_{B,1}&\lesssim   r_B^{-n}\int_{\mathbb R_+}   \int_B \big|\mathcal M_{\Delta} (f_1)(x,t) \big|^2 dx dt\nonumber\\
    &\lesssim  r_B^{-n} \iint_{\mathbb R_+^{n+1}} \big|f_1(x,t)\big|^2 dxdt\\
    &\lesssim  \left\|f\right\|_{{\mathbb L}^2(\mathbb R_+;{\rm  BMO})}^2. \nonumber
\end{align}

 We now consider the remaining term $\mathbf J_{B,2}$. By Poincar\'e's inequality again,
\begin{align*}
   \mathbf J_{B,2}&\lesssim\, \int_{\mathbb R_+}   r_B^{-n+2}  \int_B  \big|\nabla_x \mathcal M_{\Delta}  (f_2)(x,t)\big|^2 dxdt\\
  & \leq \,  \left\{   \int_{0}^{(8r_B)^2} + \int_{(8r_B)^2}^{\infty}\right\}   r_B^{-n+2}\int_B   \bigg(\int_0^t \left|\nabla_x \Delta e^{(t-s)\Delta} f_2(\cdot, s) (x)\right| ds\bigg)^2  dxdt  \\
   &=:\, \mathbf J_{B,2}^{(1)}+\mathbf J_{B,2}^{(2)}.
\end{align*}

For the first term $\mathbf J_{B,2}^{(1)}$,  
\begin{align}\label{eqn:classical-JB21}
    \mathbf J_{B,2}^{(1)} &\lesssim \int_0^{(8r_B)^2}  r_B^{-n+4} \int_B   \int_0^t      \left| \nabla_x \Delta e^{(t-s)\Delta} f_2(\cdot, s) (x) \right|^2 dsdxdt \nonumber\\
    &\lesssim \int_0^{(8r_B)^2}  r_B^{-n+4} \int_B   \int_0^t   \bigg(\int_{(4B)^c} \frac{1}{|x-y|^{n+3}}   \left|f(y,s)-\left(f(\cdot, s)\right)_{4 B}\right| dy  \bigg)^2   dsdxdt  \nonumber\\
    &\lesssim    \int_0^{(8r_B)^2}  r_B^{4}  \int_0^{(8r_B)^2}   \bigg(   \frac{1}{r_B^3}      \sum_{m=1}^\infty  \frac{1}{4^{3m}}   \frac{1}{|4^{m+1}B|}\int_{4^{m+1}B} \big|f(z,s)-(f(\cdot,s))_{4^{m+1}B} \big|\, dz    \bigg)^2   dtds \nonumber\\
    &\lesssim    \sum_{m=1}^\infty      \frac{1}{4^{3m}}     \int_0^{(8r_B)^2}      \frac{1}{|4^{m+1}B|}\int_{4^{m+1}B} \big|f(z,s)-(f(\cdot,s))_{4^{m+1}B} \big|^2 dz    ds\\
    &\lesssim  \left\|f\right\|_{{\mathbb L}^2(\mathbb R_+;{\rm  BMO})}^2. \nonumber
\end{align}

To estimate $\mathbf J_{B,2}^{(2)}$, we rewrite 
\begin{align*}
    \mathbf J_{B,2}^{(2)}&\lesssim   \int_{(8r_B)^2}^{\infty}    r_B^{-n+2}\int_B   \bigg(\int_{t-(4r_B)^2}^{t}  \left|\nabla_x \Delta e^{(t-s)\Delta} f_2(\cdot, s) (x)\right| ds\bigg)^2  dxdt  \\
    &  \quad + \int_{(8r_B)^2}^{\infty}    r_B^{-n+2}\int_B   \bigg(\int_0^{t-(4r_B)^2} \left|\nabla_x \Delta e^{(t-s)\Delta} f_2(\cdot, s) (x)\right| ds\bigg)^2  dxdt  \\
    &=:   \mathbf J_{B,2}^{(2,1)}+\mathbf J_{B,2}^{(2,2)}.
\end{align*}
 
 Note that 
 \begin{align*}
 \mathbf J_{B,2}^{(2,1)}&\lesssim   \int_{(8r_B)^2}^{\infty}   r_B^{-n+4}\int_B \int_{t-(4r_B)^2}^{t}\left| \nabla_x \Delta e^{(t-s)\Delta} f_2(\cdot, s) (x)\right|^2 dsdxdt\\
 &\lesssim   \int_{(8r_B)^2-(4r_B)^2}^{\infty}   r_B^{-n+4} \int_{0}^{(4r_B)^2}\int_B\left| \nabla_x \Delta e^{\tau\Delta} f_2(\cdot, s) (x)\right|^2 d\tau dxds,
 \end{align*}
hence one may apply an argument  similar to that of  $\mathbf I_{B,2}^{(0)} $ in {\it Step I} to obtain
\begin{align}\label{eqn:classical-JB221}
     \mathbf J_{B,2}^{(2,1)} &\lesssim   \sum_{m=1}^\infty    \frac{1}{4^{3m}}   \int_{\mathbb R_+}      \frac{1}{|4^{m+1}B|}\int_{4^{m+1}B} \big|f(z,s)-(f(\cdot,s))_{4^{m+1}B} \big|^2 dzds \\
     &\lesssim \|f\|_{{\mathbb L}^2(\mathbb R_+;{\rm BMO})}^2. \nonumber
\end{align}

It remains to estimate $\mathbf J_{B,2}^{(2,2)}$. Since for each $s,t>0$ with $t>s$ and for each $x\in B$,
\begin{align*}
      \left|\nabla_x \Delta e^{(t-s)\Delta} f_2(\cdot,s ) (x)\right| &\lesssim   \int_{(4B)^c}  \frac{1}{(t-s)^{\frac{\theta}{2}}} \frac{1}{|x-y|^{n+3-\theta}} \left|f(y,s)-\left(f(\cdot, s)\right)_{4 B}\right| dy  \\
      &\lesssim  \frac{1}{(t-s)^{\frac{\theta}{2}}}   \frac{1}{r_B^{3-\theta}}      \sum_{m=1}^\infty  \frac{1}{4^{(3-\theta)m}}   \frac{1}{|4^{m+1}B|}\int_{4^{m+1}B} \big|f(z,s)-(f(\cdot,s))_{4^{m+1}B} \big|dz.
\end{align*}
for any $\theta\in (0,3)$.  Now, we fix some $\theta\in (2,3)$ and denote
$$
     \mathbf{Avg}_{B}(s):= \sum_{m=1}^\infty  \frac{1}{4^{(3-\theta)m}}   \frac{1}{|4^{m+1}B|}\int_{4^{m+1}B} \big|f(z,s)-(f(\cdot,s))_{4^{m+1}B} \big|dz,
$$
then
\begin{align*}
    \mathbf J_{B,2}^{(2,2)}&\lesssim   \int_{(8r_B)^2}^{\infty}    r_B^{2}   \frac{1}{r_B^{6-2\theta}}   \bigg|\int_0^{t-(4r_B)^2}       \frac{1}{(t-s)^{\frac{\theta}{2}}}   \mathbf{Avg}_{B}(s) \,ds\bigg|^2    dt\\
    &=r_B^{2\theta-4}    \int_{\mathbb R_+}        \bigg|\int_{\mathbb R_+}       K(t,s)  \mathbf{Avg}_{B}(s) \,ds\bigg|^2    dt,
\end{align*}
where
$$
    K(t,s)= \frac{1}{(t-s)^{\frac{\theta}{2}}}  \mathsf 1_{((8r_B)^2,\, +\infty)}(t)\cdot \mathsf 1_{(0,\, t-(4r_B)^2)}(s).
$$

Due to $\theta\in (2,3)$ fixed, we have 
$$
     \sup_s \int_{\mathbb R_+} K(t,s) \,dt\leq \int_{(4r_B)^2}^{\infty} \frac{1}{\tau^{\frac{\theta}{2}}} \, d\tau\lesssim r_B^{2-\theta}
$$
and similarly
$$
    \sup_t \int_{\mathbb R_+} K(t,s) \,ds\lesssim r_B^{2-\theta}.
$$
Hence by Schur's lemma  (see \cite[p. 589]{Gra} for instance), the operator $T$ defined by 
$$
T(\mathbf{Avg}_{B})(t)= \int_{\mathbb R_+}       K(t,s)  \mathbf{Avg}_{B}(s) \,ds
$$ 
is bounded from $L^2(\mathbb R_+)$ to $L^2(\mathbb R_+)$ with norm $\|T\|_{L^2(\mathbb R_+)\to L^2(\mathbb R_+)}\lesssim r_B^{2-\theta}$. This yields for any fixed $\theta\in (2,3)$, 
\begin{align}\label{eqn:classical-JB222}
    \mathbf J_{B,2}^{(2,2)}&\lesssim r_B^{2\theta-4}  r_B^{2\cdot (2-\theta)}  \int_{\mathbb R_+}  \left( \mathbf{Avg}_{B}(s) \right)^2 ds \nonumber \\
    &\lesssim    \sum_{m=1}^\infty  \frac{1}{4^{(3-\theta)m}}  \int_{\mathbb R_+}     \frac{1}{|4^{m+1}B|}\int_{4^{m+1}B} \big|f(z,s)-(f(\cdot,s))_{4^{m+1}B} \big|^2 dzds\\
   & \lesssim \left\|f\right\|_{{\mathbb L}^2(\mathbb R_+;{\rm  BMO})}^2. \nonumber
\end{align}

Therefore, \eqref{eqn:MaxReg-II-classical} holds. Then \eqref{eqn:MaxReg-III-classical} follows readily   by noting
$$
   \nabla_x u(x,t)=\int_0^t  \nabla_x e^{(t-s)\Delta}f(\cdot, s)(x)\, ds=\mathcal T_{\Delta}  f(x,s)
$$
and
$$
 \partial_t u(x,t)=f(x,t)+\int_0^t  \partial_t e^{(t-s)\Delta }f(\cdot, s)(x)\, ds=f(x,t)+\mathcal M_{\Delta} (f)(x,s).
$$

We complete the proof of Theorem \ref{thm:endpoint-HE}.
\end{proof}

\smallskip

Obviously our Theorem \ref{thm:endpoint-HE} is a global version of the result in \cite{OS2}. Furthermore,
as a straightforward consequence of our argument for Theorem \ref{thm:endpoint-HE}, we have the following Carleson-type estimates.

\begin{corollary}\label{coro:compare-1}
For every $f\in  \widetilde{L^2} (\mathbb R_+;{\rm BMO}(\mathbb R^n))$,  we have
\begin{equation}\label{eqn:Carleson-I-classical}
	    \sup_{B=B(x_B,r_B)}   \bigg( r_B^{-n}\int_0^{r_B^2}\int_B
	 \big|\mathcal M_{\Delta}  (f)(x,t)\big|^2 dxdt\bigg)^{1/2}  \leq C \|f\|_{ \widetilde{L^2}(\mathbb R_+; {\rm BMO})}
\end{equation} 
and
\begin{equation}\label{eqn:Carleson-II-classical}
     \sup_{B=B(x_B,r_B)}  \bigg(r_B^{-n}\int_0^{r_B^2}\int_B \big| \mathcal T_{\Delta}  (f)(x,t)\big|^2 \frac{dxdt}{t}\bigg)^{1/2} \leq C \|f\|_{ \widetilde{L^2} (\mathbb R_+; {\rm BMO})}.
\end{equation}
\end{corollary}

\smallskip

It's clear that 
$$
    \text{ LHS\ of\ } \eqref{eqn:IVP-heat-I-classical}\lesssim \text{ LHS\ of\ } \eqref{eqn:Carleson-I-classical}\qquad   \text{ LHS\ of\ } \eqref{eqn:IVP-heat-II-variant-classical}\lesssim \text{ LHS\ of\ } \eqref{eqn:Carleson-II-classical},
$$
hence the above Carleson-type estimates \eqref{eqn:Carleson-I-classical} and   \eqref{eqn:Carleson-II-classical} are sharper than the local maximal regularity estimates \eqref{eqn:IVP-heat-I-classical} and \eqref{eqn:IVP-heat-II-variant-classical} in \cite{OS2}, respectively. This fact implies that our method is feasible and applicable to obtain refined estimates even at the local scale.

\begin{proof}[Proof of Corollary \ref{coro:compare-1}]
For any fixed $B=B(x_B,r_B)$,  we split $f=f_1+f_2+f_3$ as in the  Proof of Theorem~\ref{thm:endpoint-HE} while constraining the time variable on $(0, r_B^2)$, that is,
\begin{align*}
    f(x, s)=&\left[f(x, s)-\left(f(\cdot, s)\right)_{4 B}\right] \mathsf 1_{4 B} (x) \cdot \mathsf 1_{(0,\, r_B^2)}(s)\\
    & \  + \left[f(x,s)-\left(f(\cdot, s)\right)_{4B}\right] \mathsf 1_{  (4 B)^c}(x)\cdot \mathsf 1_{(0,\, r_B^2)}(s)+\left(f(\cdot, s)\right)_{4 B} \cdot \mathsf 1_{(0,\, r_B^2)}(s)\\
    =: &f_1(x, s)+f_{2}(x, s)+f_3(x, s).
\end{align*}

\noindent{\it Step I.} We begin by proving \eqref{eqn:Carleson-II-classical} and let
$$
    \mathsf I_{B,j} :=  r_B^{-n}\int_0^{r_B^2}\int_B \big| \mathcal T_{\Delta}  f_j (x,t)\big|^2 \frac{dxdt}{t},\quad j=1,2,3.
$$

Similar to \eqref{eqn:classical-IB1},
$$
      \mathsf I_{B,1}\lesssim  r_B^{-n} \, \|f_1\|_{L^2(\mathbb R_+^{n+1})}^2= r_B^{-n}\int_0^{r_B^2} \int_{4B}  \big|f(x,s)-(f(\cdot,s))_{4B}\big|^2dxds \lesssim \|f\|_{\widetilde{L}^2(\mathbb R_+; {\rm BMO})}^2.
$$

Using the argument for $\mathbf I_{B,2}^{(0)} $ in Thereom ~\ref{thm:endpoint-HE},  we have 
\begin{align*}
     \mathsf I_{B,2} &\lesssim r_B^{-n+2}\int_0^{r_ B^2}\int_B \int_0^{r_B^2} \bigg|\int_{(4B)^c} \frac{1}{|x-y|^{n+1}}  \left|f(y,s)-\left(f(\cdot, s)\right)_{4 B}\right| dy  \bigg|^2 dxdtds\\
     &\lesssim \sum_{m=1}^\infty \frac{1}{4^m}\int_0^{r_B^2}  \frac{1}{|4^{m+1}B|}\int_{4^{m+1}B}  \big|f(z,s)-(f(\cdot,s))_{4^{m+1}B} \big|^2 dzds \\
     & \lesssim \|f\|_{\widetilde{L}^2(\mathbb R_+; {\rm BMO})}^2.
\end{align*}

Hence \eqref{eqn:Carleson-II-classical}  holds by noting $ \mathsf I_{B,3}=0$.

\medskip

\noindent{\it Step II.} It remains to show \eqref{coro:compare-1}.
Let
$$
    \mathsf J_{B,j} :=  r_B^{-n}\int_0^\infty   \int_B \big|\mathcal M_{\Delta} (f_j)(x,t)\big|^2 dx dt ,\quad j=1,2,3,
$$
where $\mathsf J_{B,3} =0$.

Similar to \eqref{eqn:classical-JB21}, using the $L^2(\mathbb R_+^{n+1})$ boundedness of $\mathcal M_{\Delta}$, 
$$
    \mathsf J_{B,1}\lesssim r_B^{-n} \, \|f_1\|_{L^2(\mathbb R_+^{n+1})}^2  \lesssim \|f\|_{\widetilde{L}^2(\mathbb R_+; {\rm BMO})}^2.
$$ 

Using the argument for $\mathbf J_{B,2}^{(1)}$ in Thereom ~\ref{thm:endpoint-HE},  we have 
\begin{align*}
\mathsf J_{B,2} & \lesssim \int_0^{(8r_B)^2}  r_B^{-n} \int_B   t\int_0^t      \left| \Delta e^{(t-s)\Delta} (f_2(\cdot, s)) (x) \right|^2 dsdxdt \nonumber\\
&\lesssim  \sum_{m=1}^\infty      \frac{1}{4^{2m}}     \int_0^{(8r_B)^2}      \frac{1}{|4^{m+1}B|}\int_{4^{m+1}B} \big|f(z,s)-(f(\cdot,s))_{4^{m+1}B} \big|^2 dz    ds\\
         & \lesssim \|f\|_{\widetilde{L}^2(\mathbb R_+; {\rm BMO})}^2.
\end{align*}

We complete the proof of Corollary \ref{coro:compare-1}.
\end{proof}

\smallskip

One will see that the analogous  Carleson-type estimates  associated to the Schr\"odinger operator $\mathcal L$ for external forces are also valid; see 
Theorem \ref{thm:mainA} in Section \ref{sec:local-BMOL}.

\begin{remark}\label{rem:compare-2}
In comparison with the local results in \cite{OS2}, the improvement of our Theorem \ref{thm:endpoint-HE} arises from taking full advantage of the mean oscillation occurring from the definition, which deduces an additional derivative via Poincar\'e's inequality and further contributes a higher decay order which is necessary for global estimates.  

Technically, the conservation law of $e^{t\Delta }1=1$ for any $t>0$ yields the cancellations $\nabla_x e^{t\Delta}1=0$ and  $\Delta e^{t\Delta }1=0$, which deduce  the maximal regularity estimate concerning $f_3$ is trivial when splitting $f=f_1+f_2+f_3$ as done in the proof.  Furthermore, one will see in Section \ref{sec:global-BMOL} that global maximal regularity associated to Schr\"odinger operators is still valid without the conservation law,  and the corresponding argument inspired by the above proof of Theorem \ref{thm:endpoint-HE}  is rather more subtle due to an increased complexity.
\end{remark}

\smallskip

\begin{remark}
Similarly, one may define the space ${\mathbb L}^2(\mathbb R_+;{\rm CMO}(\mathbb R^n))$ in terms of limiting behaviors of mean oscillation:  a measurable function $f$ of ${\mathbb L}^2(\mathbb R_+; {\rm BMO}(\mathbb R^n))$ is in $ {\mathbb L}^2(\mathbb R_+; {\rm CMO}(\mathbb R^n))$, if $f$ satisfies  $ \gamma_i(f)=0$ for $1\leq i\leq 3$, where
\begin{align*}
 \gamma_1(f) &=\lim _{a \rightarrow 0} \sup _{B: \,r_{B} \leq a}\bigg( r_B^{-n} \int_{\mathbb R_+}\int_{B}\left|f(x,s)-(f(\cdot, s))_{B}\right|^{2} d xds\bigg)^{1 / 2} ;\\
 \gamma_2(f) &=\lim _{a \rightarrow \infty} \sup _{B: \,r_{B} \geq a}\bigg( r_B^{-n} \int_{\mathbb R_+}\int_{B}\left|f(x,s)-(f(\cdot, s))_{B}\right|^{2} d xds\bigg)^{1 / 2}  ;\\
\gamma_3(f) &= \lim _{a \rightarrow \infty} \sup _{B: \, B \subseteq (B(0, a))^c} \bigg( r_B^{-n} \int_{\mathbb R_+}\int_{B}\left|f(x,s)-(f(\cdot, s))_{B}\right|^{2} d xds\bigg)^{1 / 2}.
\end{align*}

Based on Theorem \ref{thm:endpoint-HE}, we have the following result readily.

For every $f\in {\mathbb L}^2(\mathbb R_+;{\rm CMO}(\mathbb R^n))$,  the estimates \eqref{eqn:MaxReg-II-classical}, \eqref{eqn:MaxReg-I-classical},  and \eqref{eqn:MaxReg-III-classical}  in Theorem \ref{thm:endpoint-HE} hold. Additionally,
\begin{equation}\label{eqn:MaxReg-IV}
	 \lim _{a \rightarrow 0}\sup _{B: r_{B} \leq a}  \,\mathsf{C}(f)_{B,j}
     =  \lim _{a \rightarrow \infty}\sup _{B: r_{B} \geq a}  \,\mathsf{C}(f)_{B,j}
     =   \lim _{a \rightarrow \infty}\sup _{B: B \subseteq \left(B(0, a)\right)^c}  \,\mathsf{C}(f)_{B,j}=0,\quad j=1,2,
\end{equation}
where
$$
\mathsf{C}(f)_{B,1}=\bigg( r_B^{-n} \int_{\mathbb R_+}\int_B
	 \left|\mathcal M_{\Delta} (f)(x,t)-\big(\mathcal M_{\Delta} (f)(\cdot,t)\big)_B\right|^2 dxdt\bigg)^{1/2}
$$
and 
$$
    \mathsf{C}(f)_{B,2}=\left( r_B^{-n} \int_{\mathbb R_+}\int_B \left|\mathcal T_{\Delta}(f)(x,t)-\big(\mathcal T_{\Delta} (f)(\cdot,t)\big)_B\right|^2 \frac{dxdt}{t}\right)^{1/2} .
$$
\end{remark}

\medskip


\section{Carleson type estimates and Maximal regularity in ${\rm BMO}_{\mathcal L}$: a local version} \label{sec:local-BMOL}
\setcounter{equation}{0}

As an application of our methodology for global maximal regularity of the classical heat equation in $\rm BMO$, we will extend to the study of this topic for the Cauchy problem \eqref{eqn:IVP-heat} of the Schr\"odinger equation in the ${\rm BMO}_{\mathcal L} $ setting (i.e., Theorem \ref{thm:improved-mainA}). To this end, we begin by providing a  local result in this section, which will be useful for our purpose.

Similar to  $\widetilde{L^2}(\mathbb R_+; {\rm BMO}(\mathbb R^n))$ in \eqref{eqn:L2BMO}, we say $f(x,t)$ belongs to $\widetilde{L^2}(\mathbb R_+;{\rm BMO}_{\mathcal L}(\mathbb R^n))$, if
\begin{align*} 
  \|f\|_{\widetilde{L^2}(\mathbb R_+;{\rm BMO}_{\mathcal L})} ^2   :=&  \max\left\{ \sup_{B=B(x_B,r_B):\, r_B<\rho(x_B)}   r_B^{-n} \int_{0}^{r_B^2} \int_B \left|f(x,t)- \left(f(\cdot, t)\right)_B \right|^2 dxdt,  \right.\nonumber\\ 
    & \qquad \qquad  \left.  \sup_{B=B(x_B,r_B):\, r_B\geq \rho(x_B)} r_B^{-n}  \int_{0}^{r_B^2}  \int_B |f(x,t)|^2dxdt\right\}
     <\infty.
\end{align*}
Note that $\|\cdot\|_{\widetilde{L^2}(\mathbb R_+;{\rm BMO}_{\mathcal L}) }$ is indeed a norm (which is not only a seminorm). 
Similarly, we can define $\widetilde{\dot{W}^{1,2}}(\mathbb R_+; {\rm BMO}_{\mathcal L}(\mathbb R^n))$ for $\partial_t f\in \widetilde{L^2}(\mathbb R_+;{\rm BMO}_{\mathcal L}(\mathbb R^n))$, and define $\widetilde{L^2}(\mathbb R_+; \dot{{\rm BMO}}^2_{\mathcal L}(\mathbb R^n))$ for $\mathcal L f\in \widetilde{L^2}(\mathbb R_+;{\rm BMO}_{\mathcal L}(\mathbb R^n))$, respectively.
Recall that  $\mathcal M_{\mathcal L} (f)$ and $\mathcal T_{\mathcal L}  (f)$ be  the regularity operators given in 
\eqref{MRO} and \eqref{MRO2} (with $\Delta$ replaced by $-\mathcal L$), respectively.
 
 Let us state our main result in this section, which can be viewed as the analog of Corollary \ref{coro:compare-1}, associated to $\mathcal L$.
 
\begin{theorem}\label{thm:mainA}   
Suppose $V \in \mathrm{RH}_q$ for some $q>n / 2$. For every $f\in  \widetilde{L^2}(\mathbb R_+;{\rm BMO}_{\mathcal L}(\mathbb R^n))$,  we have
\begin{equation}\label{eqn:MaxReg-II}
	 \sup_{B=B(x_B,r_B)}   \bigg( r_B^{-n}\int_0^{r_B^2}\int_B
	 \big|\mathcal M_{\mathcal L}  (f)(x,t)\big|^2 dxdt\bigg)^{1/2}   \leq C\|f\|_{\widetilde{L^2}(\mathbb R_+;{\rm BMO}_{\mathcal L})}.
\end{equation}
and
\begin{equation}\label{eqn:MaxReg-I}
	 \sup_{B=B(x_B,r_B)}  \bigg(r_B^{-n}\int_0^{r_B^2}\int_B \big| \mathcal T_{\mathcal L}  (f)(x,t)\big|^2 \frac{dxdt}{t}\bigg)^{1/2}  
	  \leq C\|f\|_{\widetilde{L^2}(\mathbb R_+;{\rm BMO}_{\mathcal L})}.
\end{equation}

As a consequence, for every $f\in  \widetilde{L^2}(\mathbb R_+;{\rm BMO}_{\mathcal L}(\mathbb R^n))$,  the Cauchy problem \eqref{eqn:IVP-heat} admits a unique solution $u\in \widetilde{\dot{W}^{1,2}}(\mathbb R_+; {\rm BMO}_{\mathcal L} (\mathbb R^n))\cap \widetilde{L^2}(\mathbb R_+; \dot{{\rm BMO}}_{\mathcal L}^2(\mathbb R^n))$ which satisfies
\begin{equation}\label{eqn:MaxReg-III}
    \left \|t^{-1/2}\partial_x u\right \|_{\widetilde{L^2}(\mathbb R_+;{\rm BMO}_{\mathcal L})} +\left \|\partial_ t u\right \|_{\widetilde{L^2}(\mathbb R_+;{\rm BMO}_{\mathcal L})} \leq C\|f\|_{\widetilde{L^2}(\mathbb R_+;{\rm BMO}_{\mathcal L})}.
\end{equation}
\end{theorem}

\smallskip
Note that  \eqref{eqn:MaxReg-II} and \eqref{eqn:MaxReg-I} do not involve the mean oscillation and present certain Carleson-type estimates for the external forces $f(x,t)$.

To prove Theorem \ref{thm:mainA}, we start by showing the following estimates motivated by \cite{JL}, which allow us to shift the derivative on the space variables to the time variable.

\smallskip

\begin{proposition}\label{prop:space-derivative}
Suppose $V \in \mathrm{RH}_q$ for some $q>n / 2$. Let $g \in L^2 \big(\mathbb{R}^n,(1+|x|)^{-(n+\beta)} dx\big )$ for some  $\beta>0$. Then for any ball $B=B(x_B,r_B)$, we have
\begin{equation}\label{eqn:space-der-1}
    \int_0^{r_B^2}\int_B \left|\nabla_x e^{-t \mathcal L} g\right|^2 dxdt\leq C\int_0^{(2r_B)^2}\int_{2B} \Big ( \left|\partial_t e^{-t \mathcal L} g\right|  \left |e^{-t \mathcal L} g\right| +\frac{1}{r_B^2} \left|e^{-t \mathcal L} g\right|^2 \Big )\,  dxdt.
\end{equation}
Moreover, for any constant $c_0\neq 0$,
\begin{align}\label{eqn:space-der-2}
	   & \int_0^{r_B^2}\int_B \left|\nabla_x e^{-t \mathcal L} g\right|^2 dxdt \notag \\
	   \leq \,& C\int_0^{(2r_B)^2}\int_{2B} \bigg ( \left|\partial_t e^{-t \mathcal L} g\right|  \left |e^{-t \mathcal L} g-c_0\right| +\frac{1}{r_B^2} \left|e^{-t \mathcal L} g-c_0\right|^2 +\left|e^{-t\mathcal L} g\right| \left|e^{-t\mathcal L} g-c_0\right|  V\bigg )\,  dxdt.
\end{align}
\end{proposition}

\begin{proof}
The argument is similar to that of Proposition~5.2 in \cite{JL}. 
Since $g \in L^2 (\mathbb{R}^n,(1+|x|)^{-(n+\beta)} dx)$ for some  $\beta>0$, one may apply the heat estimates associted to $\mathcal L$ to see that  $e^{-t \mathcal L} g$ and  $\partial_t e^{-t \mathcal L} g$ are locally bounded. Moreover, for any compactly supported smooth function $\psi(x, t)$ on $\mathbb{R}_{+}^{n+1}$, combining integration by parts and  Young's inequality for products, we have
$$
\begin{aligned}
\iint_{\mathbb R_+^{n+1}}\left|\nabla_x e^{-t \mathcal L} g\right|^2 \psi^2 d x d t  =&\iint_{\mathbb R_+^{n+1}}\left[\left\langle\nabla_x e^{-t \mathcal L} g, \nabla_x\left(e^{-t \mathcal L} g \psi^2\right)\right\rangle-2\left\langle\nabla_x e^{-t \mathcal L} g, \nabla_x \psi\right\rangle \psi e^{-t \mathcal L} g\right] d x d t \\
\leq & \iint_{\mathbb R_+^{n+1}}  (\mathcal{L}  e^{-t \mathcal L} g)\cdot  (e^{-t \mathcal L} g \psi^2)\, d x d t-\iint_{\mathbb R_+^{n+1}} V |e^{-t \mathcal L} g|^2 \psi^2\, d x d t \\
& +\frac{1}{2} \iint_{\mathbb R_+^{n+1}}\left|\nabla_x e^{-t \mathcal L} g\right|^2 \psi^2 d x d t+2 \iint_{\mathbb R_+^{n+1}}\left|\nabla_x \psi\right|^2\left|e^{-t \mathcal L} g\right|^2 d x d t.
\end{aligned}
$$
Plugging the third term in the right hand back to the left,  noting that $\mathcal L  e^{-t \mathcal L}   g = -\partial_t e^{-t \mathcal L} g $ and the second term in the right hand is non-positive,
$$
\iint_{\mathbb R_+^{n+1}}\left|\nabla_x e^{-t \mathcal L} g\right|^2 \psi^2 d x d t \leq -2 \iint_{\mathbb R_+^{n+1}}  (\partial_t  e^{-t \mathcal L} g)\cdot  (e^{-t \mathcal L} g \psi^2)\, d x d t +4 \iint_{\mathbb R_+^{n+1}}\left|\nabla_x \psi\right|^2\left|e^{-t \mathcal L} g\right|^2 d x d t .
$$
This implies that $e^{-t \mathcal L} g \in W_{\text {loc }}^{1,2} (\mathbb R_+^{n+1})$.

Take a smooth funciton $\varphi$ on $\mathbb R^n$ with ${\rm supp\,}\varphi\subset 2B$ such that $\varphi=1$ on $B$ and $|\nabla_x \varphi|\leq C/{r_B}$.
For each $\epsilon \in (0, r_B^2)$, take a smooth function $\phi_\epsilon(t)$ on $\mathbb R$ such that $\operatorname{supp} \phi_\epsilon \subset (\epsilon, (2 r_B)^2)$ and  $\phi_\epsilon(t)=1$ on $(2 \epsilon, r_B^2)$. Without loss of generality, one may assume that $\|\varphi\|_{L^\infty}=\|\phi_\epsilon\|_{L^\infty}=1$.
The above argument shows that $(\varphi \phi_\epsilon)^2 e^{-t \mathcal L} g\in W^{1,2}(\mathbb R_+^{n+1})$ with ${\rm supp\,} (\varphi \phi_\epsilon)^2 e^{-t \mathcal L} g \subset 2B\times (\epsilon, (2r_B)^2)  $. Moreover, by integration by parts and $\mathcal L  e^{-t \mathcal L}  g = -\partial_t e^{-t \mathcal L} g $ again, for any given constant $c_0$,
\begin{align}\label{eqn:space-der-3}
	&\int_0^{(2r_B)^2}\int_{2B} \left|\nabla_x e^{-t \mathcal L} g\right|^2 (\varphi \phi_\epsilon)^2 dxdt \notag\\
	=&\int_0^{(2r_B)^2}\int_{2B} \left|\nabla_x \left(e^{-t \mathcal L} g-c_0\right)\right|^2 (\varphi \phi_\epsilon)^2 dxdt\notag\\
	=& \int_0^{(2r_B)^2}\int_{2B} \left\langle \nabla_x e^{-t \mathcal L} g, \nabla_x \left( (\varphi \phi_\epsilon)^2 \big (e^{-t \mathcal L} g-c_0\big ) \right)\right\rangle dxdt\notag\\
	& -2\int_0^{(2r_B)^2}\int_{2B} \left\langle   \nabla_x e^{-t \mathcal L} g , \nabla_x \varphi \right\rangle \varphi\phi_\epsilon^2 \big (e^{-t \mathcal L} g-c_0\big) \, dxdt\notag\\
	=& - \int_0^{(2r_B)^2}\int_{2B}  (\partial_t e^{-t \mathcal L} g) (\varphi\phi_\epsilon)^2 \big( e^{-t \mathcal L} g-c_0\big)\, dxdt -\int_0^{(2r_B)^2}\int_{2B} V e^{-t \mathcal L} g \big(e^{-t\mathcal L}g-c_0\big)  (\varphi\phi_\epsilon)^2 \, dxdt\notag\\
	& -2\int_0^{(2r_B)^2} \int_{2B} \left\langle   \nabla_x e^{-t \mathcal L} g , \nabla_x \varphi \right\rangle \varphi\phi_\epsilon^2 \big(e^{-t\mathcal L}g-c_0\big)  \, dxdt.
\end{align}

If $c_0\neq 0$. Using Young's inequality again, it follows from the assumption $\|\varphi\|_{L^\infty}=\|\phi_\epsilon\|_{L^\infty}=1$ that 
\begin{align*}
&\int_0^{(2r_B)^2}\int_{2B} \left|\nabla_x e^{-t \mathcal L} g\right|^2 (\varphi \phi_\epsilon)^2 dxdt \\
	\leq & \int_0^{ (2 r_B)^2}\int_{2B}  \left|\partial_t e^{-t \mathcal L} g\right| \left| e^{-t \mathcal L} g-c_0\right| dxdt  +\int_0^{(2r_B)^2}\int_{2B} V\left|e^{-t\mathcal L} g\right| \left|e^{-t\mathcal L} g-c_0\right| dxdt
	\\
	& +\frac{1}{2}\int_0^{(2r_B)^2 }\int_{2B} \left|\nabla _x e^{-t \mathcal L} g\right|^2 (\varphi\phi_\epsilon)^2  dxdt +2\int_0^{(2r_B)^2}\int_{2B} |\nabla_x \varphi|^2 \phi_\epsilon^2 \left|e^{-t \mathcal L} g-c_0\right|^2 dxdt,
\end{align*}
plugging the third term in the right hand back to the left, then
\begin{align*}
	&\int_0^{(2r_B)^2}\int_{2B} \left|\nabla_x e^{-t \mathcal L} g\right|^2 (\varphi \phi_\epsilon)^2 dxdt\\
	\leq\,& 2\int_0^{ (2 r_B)^2}\int_{2B}  \left|\partial_t e^{-t \mathcal L} g\right| \left| e^{-t \mathcal L} g-c_0\right| dxdt+2\int_0^{(2r_B)^2}\int_{2B} V\left|e^{-t\mathcal L} g\right| \left|e^{-t\mathcal L} g-c_0\right| dxdt\\
	&+4\int_0^{(2r_B)^2}\int_{2B} |\nabla_x \varphi|^2 \phi_\epsilon^2 \left|e^{-t \mathcal L} g-c_0\right|^2 dxdt,
\end{align*}
 hence \eqref{eqn:space-der-2} holds by noting $|\nabla_x \varphi|\leq C/{r_B}$ and  letting $\varepsilon\to 0+$, as desired.

In particular, if $c_0= 0$, it follows from \eqref{eqn:space-der-3} to see
\begin{align*}
   &\int_0^{(2r_B)^2}\int_{2B} \left|\nabla_x e^{-t \mathcal L} g\right|^2 (\varphi \phi_\epsilon)^2 dxdt \\
   \leq\,&  - \int_0^{(2r_B)^2}\int_{2B}  (\partial_t e^{-t \mathcal L} g) (\varphi\phi_\epsilon)^2  e^{-t \mathcal L} g\, dxdt-2\int_0^{(2r_B)^2} \int_{2B} \left\langle   \nabla_x e^{-t \mathcal L} g , \nabla_x \varphi \right\rangle \varphi\phi_\epsilon^2 e^{-t\mathcal L}g  \, dxdt.
 \end{align*}
 Using the same argument above, we obtain \eqref{eqn:space-der-1}.
\end{proof}

\smallskip

\begin{remark}\label{rem:homo-Carleson}
The constant $c_0\neq 0$ in \eqref{eqn:space-der-2} can be relaxed to any measurable function $c_0(t)\not \equiv 0$ on $(0,(2r_B)^2)$, since the whole argument of Proposition \ref{prop:space-derivative} only concerns $c_0$ itself and its space derivative.
On the other hand, the proof also allows us to relax the time interval in \eqref{eqn:space-der-1} and \eqref{eqn:space-der-2} from $(0,r_B^2)$ to  $(0,T)$ for any $T\leq \infty$, due to the fact that the proof does not involve the (time) derivative of the auxiliary function $\phi_\epsilon$. 
Similarly, when the time interval in the left-hand side of \eqref{eqn:space-der-1} and \eqref{eqn:space-der-2} is some $\big( (4^kr_B)^2, (4^{k+1}r_B)^2\big)$ for $k\geq 1$, then the time interval in the right-hand side can be chosen as $\big( (4^kr_B)^2-2r_B^2, (4^{k+1}r_B)^2+2r_B^2\big)$.
\end{remark}

\smallskip

As a consequence of Proposition~\ref{prop:space-derivative}, we have the following Carleson estimates. The proof is standard and we skip the details.

\begin{proposition}\label{prop:Carleson}
Suppose $V \in \mathrm{RH}_q$ for some $q>n / 2$. Let $g\in {\rm BMO}_{\mathcal L}(\mathbb R^n)$, we have
$$
    \sup_{B=B(x_B, r_B)}\,  r_B^{-n}\int_0^{r_B^2}\int_{B(x_B,r_B)} \left( \big|t\partial_t e^{-t \mathcal L} g(x)\big|^2 +\big|\sqrt{t}\nabla_x e^{-t \mathcal L} g(x)\big|^2 \right) \frac{dxdt}{t}\leq C \|g\|_{{\rm BMO}_\mathcal L}^2.
$$
\end{proposition}

\smallskip

Based on  Proposition~\ref{prop:Carleson}, one may furthermore obtain certain Carleson-type estimates involving the inhomogeneous terms of the equation \eqref{eqn:IVP-heat}.

\begin{proposition}\label{prop:Car-BMO-2}
	Suppose $V \in \mathrm{RH}_q$ for some $q>n / 2$. Let $f\in \widetilde{L^2}(\mathbb R_+;{\rm BMO}_{\mathcal L})$, we have
\begin{equation}\label{eqn:inhomo-Carleson-1}
    \sup_{B=B(x_B,r_B)}   r_B^{-n}\int_0^{r_B^2}\int_B \int_0^t \left|\nabla_x e^{-(t-s)\mathcal L}f(\cdot, s)(x)\right|^2 ds dxdt \leq C \|f\|_{\widetilde{L^2}(\mathbb R_+;{\rm BMO}_{\mathcal L})}^2.
\end{equation}
\end{proposition}

\smallskip

\begin{remark}
Note that by Proposition~\ref{prop:Carleson},
\begin{align*}
    {\rm LHS\, of\,   }\eqref{eqn:inhomo-Carleson-1} &= r_B^{-n} \int_0^{r_B^2} \int_B \int_s^{r_B^2} \left| \nabla_x e^{-(t-s) \mathcal L} f\,(\cdot,s)(x)\right|^2  dt  dxds\\
    & \leq   r_B^{-n}  \int_0^{r_B^2}  \int_0^{r_B^2} \int_B  \left| \nabla_x e^{-\tau \mathcal L} f\,(\cdot,s)(x)\right|^2   dx d\tau ds\\
    &\leq C \int_0^{r_B^2} \|f(\cdot, s)\|_{{\rm BMO}_{\mathcal L}}^2 ds.
\end{align*}
It's rougher than the aimed \eqref{eqn:inhomo-Carleson-1}, so we need a somewhat elaborate argument to prove it.

On the other hand, it follows from   H\"older's inequality to see
$$
    {\rm LHS\, of\,  } \eqref{eqn:MaxReg-I}\leq {\rm LHS\, of\, }\eqref{eqn:inhomo-Carleson-1} ,
$$
hence \eqref{eqn:MaxReg-I} in Theorem~\ref{thm:mainA} is a straighforward consequence of  Proposition~\ref{prop:Car-BMO-2}.
\end{remark}

To prove Proposition~\ref{prop:Car-BMO-2}, we introduce some preliminary and auxiliary results.

\begin{lemma} \label{lem:size-rho}
{\rm ( \cite[Lemma~1.4]{Shen1}.)}\
Suppose $V\in {\rm RH}_q$ for some $q> n/2$. There exist $c>1$ and $k_0\geq1$ such that for all $x,y\in\mathbb{R}^n$,
\begin{equation}\label{eqn:size-rho}
c^{-1}\left(1+\frac{|x-y|}{\rho(x)}\right)^{-k_0} \rho(x)\leq\rho(y)\leq c
\left(1+\frac{|x-y|}{\rho(x)}\right)^{\frac{k_0}{k_0+1}} \rho(x).
\end{equation}
In particular, $\rho(x)\approx \rho(y)$ when $y\in B(x, r)$ and $r\lesssim\rho(x)$.
\end{lemma}

\smallskip

\begin{lemma}\label{lem:heat-Schrodinger}
{\rm (see \cite{DGMTZ}.)} ~~ 
Suppose $V\in {\rm RH}_q$ for some $q> n/2$. Let ${\mathcal K}_t(x,y)$ be the kernel of the semigroup $\{e^{-t\mathcal L}\}_{t>0}$. For every $0<\delta<\min\Big\{1,2-\frac{n}{q}\Big\}$ and integer $m\geq 0$, there exist constants $c>0$  such that for every $N>0$ there is a constant $C_N$ such that 
 \begin{itemize}
\item[(i)]
$\displaystyle  \left| t^m \partial_t^m {\mathcal K}_t(x,y)\right|\leq C_N \, t^{-\frac{n}{2}}\exp\left(- \frac{|x-y|^2}{ct}\right)  \left(1+\frac{\sqrt t}{\rho(x)}+\frac{\sqrt t}{\rho(y)}\right)^{-N} $.

\item[(ii)] $\displaystyle \left| t^m\partial_t^m e^{-t \mathcal{L}}(1)(x)\right| \leq C \left(\frac{\sqrt{t}}{\rho(x)}\right)^\delta \left(1+\frac{\sqrt{t}}{\rho(x)}\right)^{-N}$ with $m\geq 1$.

\item[(iii)] for all $|h|\leq \sqrt{t}$,
$$
  \big| t \partial_t {\mathcal K}_t(x+h,y)-t \partial_t {\mathcal K}_t(x,y)\big|\leq C_N \,  \left(\frac{|h|}{\sqrt{t}}   \right)^\delta t^{-\frac{n}{2}}\exp\left(- \frac{|x-y|^2}{ct}\right)  \left(1+\frac{\sqrt t}{\rho(x)}+\frac{\sqrt t}{\rho(y)}\right)^{-N}. 
$$ 
 \end{itemize}
\end{lemma}

\smallskip

For $s>0$, we define
$$
\mathbb{F}(s):=\left\{\psi: \mathbb{C} \rightarrow \mathbb{C} \text { measurable }: \quad|\psi(z)| \leq C \frac{|z|^s}{\left(1+|z|^{2 s}\right)}\right\}.
$$
Then for any non-zero function $\psi \in \mathbb{F}(s)$, we have that $\left(\int_0^{\infty}|\psi(t)|^2 dt/t\right)^{1 / 2}<\infty$. Denote $\psi_t(z):=\psi(t z)$ for $t>0$. It follows from the spectral theory in \cite{Yo} that for any $g \in L^2\left(\mathbb{R}^n\right)$,
\begin{align}\label{eqn:FC}
 \int_0^{\infty}\|\psi(t \mathcal{L}) g\|_{L^2\left(\mathbb{R}^n\right)}^2 \frac{dt}{t}  & = \int_0^{\infty}\left\langle\bar{\psi}(t \mathcal{L}) \psi(t\mathcal{L}) g, g\right\rangle \frac{dt}{t}  
 = \left\langle\int_0^{\infty}|\psi|^2(t \mathcal{L}) \frac{dt}{t} g, g\right\rangle  \leq \kappa\|g\|_{L^2\left(\mathbb{R}^n\right)}^2,
\end{align}
where $\kappa= \int_0^{\infty}|\psi(t)|^2 dt / t$. The estimate will be useful in this article.

\smallskip

\begin{proof}[Proof of Proposition~\ref{prop:Car-BMO-2}]
For any given ball $B=B(x_B,r_B)\subset \mathbb R^n$, we split $f=f_1+f_2+f_3$ as in the  proof of Corollary~\ref{coro:compare-1}. 
Denote
\begin{align*}
\mathcal I_{B, j}:=    r_B^{-n}  \int_0^{r_B^2} \int_0^{r_B^2} \int_B  \left| \nabla_x e^{-\tau \mathcal L} f_j\,(\cdot,s)(x)\right|^2   dx d\tau ds,
\quad j=1,2,3,
\end{align*}
then
$$
 r_B^{-n}   \int_0^{r_B^2} \int_B \int_0^t \left|\nabla_x e^{-(t-s)\mathcal L}f(\cdot, s)(x)\right|^2 ds dxdt  \lesssim \sum_{j=1}^3  \mathcal I_{B, j} .
$$

Similar to the estimate on  $\mathbf I_{B,1}$ in Theorem~\ref{thm:endpoint-HE},
we apply the $L^2$-boundedness of the Riesz transform $\nabla_x \mathcal L^{-1/2}$ (see \cite[Theorem~0.4]{Shen1}) and \eqref{eqn:FC} to obtain
\begin{align*}
	\mathcal I_{B,1} &=  r_B^{-n}  \int_0^{r_B^2}  \int_0^{r_B^2} \int_{B} \left|\nabla_x \mathcal L^{-1/2}  \sqrt{\mathcal L} e^{-\tau \mathcal L} f_1(\cdot, s)(x)\right|^2 dxd\tau ds\\
	&\lesssim    r_B^{-n}  \int_0^{r_B^2}   \int_0^{r_B^2} \int_{\mathbb R^n} \left|  \sqrt{\tau\mathcal L} e^{-\tau\mathcal L} f_1(\cdot, s)(x)\right|^2  \frac{dxd\tau}{\tau} ds \\
	&\lesssim   r_B^{-n}  \int_0^{r_B^2}     \int_{4B}\left|f(x,s)-\left(f(\cdot,s)\right)_{4B}\right|^2 dx   ds\\
	&\lesssim \|f\|_{\widetilde{L^2}(\mathbb R_+;{\rm BMO}_{\mathcal L})}^2.
\end{align*}

Consider the second term $\mathcal I_{B,2}$. By \eqref{eqn:space-der-1} in  Proposition~\ref{prop:space-derivative},
\begin{equation}\label{eqn:J2-new}
	\mathcal I_{B,2} \lesssim    r_B^{-n}    \int_0^{r_B^2}  \int_0^{(2r_B)^2} \int_{2B}  \left( \left|\partial_t e^{-\tau \mathcal L} f_2(\cdot,s)\right|  \left|e^{-\tau \mathcal L} f_2(\cdot,s)\right| +\frac{1}{r_B^2} \left|e^{-\tau \mathcal L} f_2(\cdot,s)\right|^2 \right) dxd\tau ds.
\end{equation}

The following argument for $\mathcal I_{B,2}$ is similar to that of 
$ \mathbf I_{B,2}^{(0)}$ in the proof of Theorem~\ref{thm:endpoint-HE}. Concretely,
it follows from Lemma~\ref{lem:heat-Schrodinger}  that for any $x\in 2B$,  $0<\tau<(2r_B)^2$ and for fixed $s$,
\begin{align}\label{eqn:der-time-heat}
    \left|\partial_t e^{-\tau \mathcal L} f_2(\cdot,s) (x) \right|  & \lesssim  \sum_{j=1}^\infty \int_{4^{j+1}B\setminus {4^j B}} \frac{1}{|x-y|^{n+2}} \Big ( \left|f(y,s)-\left(f(\cdot, s)\right)_{4^{j+1}B}\right| +\left|\left(f(\cdot,s)\right)_{4^{j+1}B}-\left(f(\cdot, s)\right)_{4B}\right|\Big )\, dy.  \nonumber\\
  & \lesssim \frac{1}{r_B^2} \sum_{k=1}^\infty  \left( \sum_{j=k}^\infty   \frac{1}{4^{2j}} \cdot \frac{1}{|4^k B|} \int_{4^k B} \left|f(y,s)-\left(f(\cdot, s)\right)_{4^k B}\right| dy \right)\nonumber\\
    &\lesssim \frac{1}{r_B^2} \sum_{k=1}^\infty  \frac{1}{4^{2k}}\cdot \frac{1}{|4^k B|} \int_{4^k B} \left|f(y,s)-\left(f(\cdot, s)\right)_{4^k B}\right| dy.
\end{align}
Similarly,
$$
    \left|e^{-\tau \mathcal L} f_2(\cdot, s)(x)\right|\lesssim \frac{\tau}{r_B^2} \sum_{k=1}^\infty  \frac{1}{4^{2k}}\cdot \frac{1}{|4^k B|} \int_{4^k B} \left|f(y,s)-\left(f(\cdot, s)\right)_{4^k B}\right| dy.
$$
Thus
\begin{align}\label{eqn:I2-aux-new}
    &\left|\partial_t e^{-\tau \mathcal L} f_2(\cdot,s)(x)\right|  \left|e^{-\tau \mathcal L} f_2(\cdot,s)(x)\right| +\frac{1}{r_B^2} \left|e^{-\tau \mathcal L} f_2(\cdot,s)(x)\right|^2	\nonumber\\
    \lesssim\, &  \left(\frac{\tau}{r_B^4} +\frac{\tau^2}{r_B^6}\right)  \left( \sum_{k=1}^\infty  \frac{1}{4^{2k}}\cdot \frac{1}{|4^k B|} \int_{4^k B} \left|f(y,s)-\left(f(\cdot, s)\right)_{4^k B}\right| dy\right)^2  \nonumber\\
     \lesssim\, & \left(\frac{\tau}{r_B^4} +\frac{\tau^2}{r_B^6}\right)    \sum_{k=1}^\infty  \frac{1}{4^{2k}}\cdot \frac{1}{|4^k B|} \int_{4^k B} \left|f(y,s)-\left(f(\cdot, s)\right)_{4^k B}\right|^2 dy.
\end{align}
 
Hence,  we have by \eqref{eqn:J2-new} that
\begin{align*}
    \mathcal I_{B,2} &\lesssim r_B^{-n} \int_0^{r_B^2} 	\sum_{k=1}^\infty  \frac{1}{4^{2k}}\cdot \frac{1}{|4^k B|} \int_{4^k B} \left|f(y,s)-\left(f(\cdot, s)\right)_{4^k B}\right|^2 dyds\\
   &\lesssim  \|f\|_{\widetilde{L^2}(\mathbb R_+;{\rm BMO}_{\mathcal L})}^2 .
\end{align*}

It remains to estimate  $\mathcal I_{B,3}$, and we consider it in the following two cases.

\smallskip
\noindent{\it Case 1.} $4r_B\geq \rho(x_B)$. \  \
Combining  Lemma~\ref{lem:heat-Schrodinger} and \eqref{eqn:space-der-1} in Proposition~\ref{prop:space-derivative} and  again, we obtain that for  every $0<\delta<\min\Big\{1,2-\frac{n}{q}\Big\}$, 
\begin{align*}
	\mathcal I_{B,3}&\lesssim    r_B^{-n}\int_0^{r_B^2} \left|\left(f(\cdot, s)\right)_{4B}\right|^2 \int_0^{(2r_B)^2} \int_{2B}  \left( \left|\tau\partial_t e^{-\tau \mathcal L} (1)\right|  \left|e^{-\tau \mathcal L} (1)\right| +\frac{\tau}{r_B^2} \left|e^{-\tau \mathcal L} (1)\right|^2 \right) \frac{dxd\tau}{\tau} ds \\
	&\lesssim  r_B^{-n}  \int_0^{r_B^2}  \big|f(\cdot, s)\big|_{4B} ^2   \int_0^{(2r_B)^2} \int_{2B}   \left[\left(\frac{\sqrt{\tau}}{\rho(x)}\right)^{\delta }\left(1+\frac{\sqrt{\tau}}{\rho(x)}\right)^{-N}     +\frac{\tau}{r_B^2} \right]  \frac{dxd\tau}{\tau} ds  \\
	&\lesssim   r_B^{-n} \int_0^{r_B^2}   \frac{1}{|4B|} \int_{4B} |f(y,s)|^2  dy ds \cdot      \int_{2B}\left[\int_0^\infty  \frac{\nu^{\delta -1}}{(1+\nu)^N} d\nu   +1  \right] dx  \\
	&\lesssim   \|f\|_{\widetilde{L^2}(\mathbb R_+;{\rm BMO}_{\mathcal L})}^2.
\end{align*}

\smallskip

\noindent{\it Case 2.} $4r_B< \rho(x_B)$.  \  \
For any fixed $s$,
by \eqref{eqn:space-der-2} in Proposition~\ref{prop:space-derivative} and taking $c_0= \left(f(\cdot, s)\right)_{4B}$, one may combine the conservation property $e^{\tau\Delta }(1)= 1$ to obtain
\begin{align}\label{eqn:J3-new}
	\mathcal I_{B,3}\lesssim \, & r_B^{-n}\int_0^{r_B^2}    \int_0^{(2r_B)^2}\int_{2B}  \left|\partial_t e^{-\tau \mathcal L}  f_3(\cdot, s)(x)\right|  \left |e^{-\tau \mathcal L} f_3(\cdot, s) (x)-e^{\tau \Delta} f_3(\cdot, s)  (x)\right| dxd\tau ds \nonumber\\
	& +  r_B^{-n} \int_0^{r_B^2}    \int_0^{(2r_B)^2}\int_{2B}  \frac{1}{r_B^2} \left|e^{-\tau \mathcal L} f_3(\cdot, s) (x)-e^{\tau\Delta } f_3(\cdot, s)(x)\right|^2\,dxd\tau ds\nonumber\\
	& + r_B^{-n}\int_0^{r_B^2}    \int_0^{(2r_B)^2}\int_{2B}  \left|e^{-\tau \mathcal L} f_3(\cdot, s)(x)\right| \left|e^{-\tau \mathcal L} f_3(\cdot, s)(x)-e^{\tau \Delta }f_3(\cdot, s)(x)\right|  V\,  dxd\tau ds.
\end{align}

Note that  $4r_B<\rho(x_B)$, denote $\kappa=\lfloor \log_2 \frac{\rho(x_B)}{r_B} \rfloor +1$, where $\lfloor \cdot \rfloor$ denotes the floor function. Therefore,  we have $2^{\kappa }r_B>\rho(x_B)$,  and
\begin{align*}
     \left|\left(f(\cdot,s)\right)_{4B}\right| &\leq    \left|\left(f(\cdot,s)\right)_{2^{\kappa}B}\right| +  \sum_{j=2}^{\kappa -1}\left|\left(f(\cdot,s)\right)_{2^{j}B}-\left(f(\cdot,s)\right)_{2^{j+1}B}\right|.
\end{align*}
Let
\begin{equation}\label{MBs}
    \mathbf M_B(s)=\frac{1}{\kappa -1}\Bigg( \left|\left(f(\cdot,s)\right)_{2^{\kappa}B}\right| +  \sum_{j=2}^{\kappa -1}\left|\left(f(\cdot,s)\right)_{2^{j}B}-\left(f(\cdot,s)\right)_{2^{j+1}B}\right|\Bigg).
\end{equation}

This, combined with  (ii) of Lemma \ref{lem:heat-Schrodinger},
\begin{align}\label{eqn:der-time-heat-2}
  \left|\partial_t e^{-\tau \mathcal L} (f(\cdot,s))_{4B}(x)\right|&\lesssim   \tau^{-1}   \left|\left(f(\cdot,s)\right)_{4B}\right| \left(\frac{\sqrt{\tau}}{\rho(x_B)}\right)^\delta  \nonumber\\
  &\lesssim  \tau^{-1}
 \big \lfloor \log_2 \frac{\rho(x_B)}{r_B} \big\rfloor     \cdot \mathbf M_B(s)
  \cdot \left(\frac{\sqrt{\tau}}{\rho(x_B)}\right)^\delta \nonumber\\
  &\lesssim 
   \tau^{-1}\left(\frac{\sqrt{\tau}}{r_B}\right)^{\delta\over 2} \left(\frac{\sqrt{\tau}}{\rho(x_B)}\right)^{\delta\over 2}  \mathbf M_B(s) .
\end{align}

Denote by $h_\tau(x)$
 the kernel of the classical heat semigroup $\big\{e^{\tau\Delta}\big\}_{\tau>0}$.  Recall that 
there exists a nonnegative Schwartz  function $\varphi$ on $\mathbb R^n$ such that
\begin{equation*}
\big| h_\tau(x-y)-{\mathcal K}_\tau(x,y) \big|\leq\left(\frac{\sqrt{\tau}}{\rho(x)}\right)^{2-\frac{n}{q}}\varphi_\tau(x-y),\quad x,y\in\mathbb R^n,~\tau>0,
\end{equation*}
where $\varphi_\tau(x)=\tau^{-n/2}\varphi\big(x/\sqrt{t}\big)$; see  \cite[Proposition~2.16]{DZ2}.
This yields that 
\begin{align*}
   \left |e^{-\tau \mathcal L} f_3(\cdot, s)(x)- e^{\tau \Delta } f_3(\cdot, s)(x)\right| &\lesssim  \big\lfloor \log_2 \frac{\rho(x_B)}{r_B} \big\rfloor     \cdot \mathbf M_B(s) \cdot \left(\frac{\sqrt{\tau}}{\rho(x_B)}\right)^{2-\frac{n}{q}} 
\end{align*}
and  combining  Lemma \ref{lem:heat-Schrodinger} to see
\begin{align*}
   &\left|e^{-\tau\mathcal L} f_3(\cdot,s)(x)\right| \left|e^{-\tau\mathcal L} f_3(\cdot,s)(x)-e^{\tau\Delta }f_3(\cdot,s)(x)\right|\\
   \lesssim\, & \left( \big\lfloor \log_2 \frac{\rho(x_B)}{r_B} \big\rfloor     \cdot \mathbf M_B(s) \right)^2  \left(\frac{\sqrt{\tau}}{\rho(x_B)}\right)^{2-\frac{n}{q}}\\
    \lesssim \, &  \left(\frac{\sqrt{\tau}}{r_B}\right)^{2-\frac{n}{q}-\frac{\delta}{2}}\left(\frac{\sqrt{\tau}}{\rho(x_B)}\right)^{\delta\over 2} \big (\mathbf M_B(s)\big ) ^2 .
\end{align*}
This, allows us to use the definition \eqref{eqn:critical-funct} of the auxiliary function $\rho$ to see when $4r_B<\rho(x_B)$,
\begin{align*}
	&r_B^{-n}\int_{2B}  \left|e^{-\tau\mathcal L} f_3(\cdot,s)(x)\right| \left|e^{-\tau\mathcal L} f_3(\cdot,s)(x)-e^{\tau\Delta }f_3(\cdot,s)(x)\right|  V(x)\,  dx \\
	\lesssim \, &       \frac{1}{r_B^2}\left(\frac{\sqrt{\tau}}{r_B}\right)^{2-\frac{n}{q}-\frac{\delta}{2}}\left(\frac{\sqrt{\tau}}{\rho(x_B)}\right)^{\delta\over 2} \big(\mathbf M_B(s)\big) ^2.
\end{align*}
Plugging these estimates back to \eqref{eqn:J3-new}   and noting  $0<\delta<\min\Big\{1,2-\frac{n}{q}\Big\}$ to see
\begin{align*}
   \mathcal I_{B,3} \lesssim  \int_0^{r_B^2}   \left(\frac{r_B}{\rho(x_B)}\right)^{\delta\over 2} \big(  \mathbf M_B(s)\big) ^2   ds
   \lesssim  \int_0^{r_B^2}   \big(  \mathbf M_B(s)\big) ^2   ds
   \lesssim  \|f\|_{\widetilde{L^2}(\mathbb R_+;{\rm BMO}_{\mathcal L})}^2,
\end{align*}
where the last inequality follows from the Cauchy-Schwarz inequality.

From the above, we complete the proof.
\end{proof}

\medskip


\begin{remark}\label{rem:JB3}
Indeed, for any  ball $B=B(x_B,r_B)$ with $4r_B< \rho(x_B)$, the above argument yields 
\begin{align*}
    \mathcal I_{B,3} \lesssim   \left(\frac{r_B}{\rho(x_B)}\right)^{\delta\over 2}  \int_0^{r_B^2}   \big(  \mathbf M_B(s)\big) ^2   ds\lesssim   \left(\frac{r_B}{\rho(x_B)}\right)^{\delta\over 2} \|f\|_{\widetilde{L^2}(\mathbb R_+;{\rm BMO}_{\mathcal L})}^2
\end{align*}
and 
\begin{align*}
     \mathcal I_{B,3} &\lesssim  \left(\frac{r_B}{\rho(x_B)}\right)^{\delta\over 2} \int_0^{r_B^2}   \big|f(\cdot,s)\big|_{4B}^2 \, ds\\
      &\lesssim  \left(\frac{r_B}{\rho(x_B)}\right)^{\delta\over 2}   \left(\frac{\rho(x_B)}{r_B}  \right)^{n }  \frac{1}{|B(x_B,\rho(x_B))|} \int_0^{\rho(x_B)^2} \int_{B(x_B,\,\rho(x_B))} |f(x,s)|^2 dxds.
\end{align*}
Hence, for any $\theta\in [0,1]$, 
\begin{align}\label{eqn:aux-JB3}
    \mathcal I_{B,3}   &\leq  C  \left(\frac{r_B}{\rho(x_B)}\right)^{\delta\over 2} \|f\|_{\widetilde{L^2}(\mathbb R_+;{\rm BMO}_{\mathcal L})}^{2\theta}    \left(\frac{\rho(x_B)}{r_B}  \right)^{n(1-\theta) }   \left( \frac{1}{|B(x_B,\rho(x_B))|} \int_0^{\rho(x_B)^2}  \int_{B(x_B,\rho(x_B))} |f(x,s)|^2 dxds\right)^{1-\theta}\nonumber \\
    &=   C  \left(\frac{r_B}{\rho(x_B)}\right)^{\frac{\delta}{ 2} -n(1-\theta)} \|f\|_{\widetilde{L^2}(\mathbb R_+;{\rm BMO}_{\mathcal L})}^{2\theta} \left( \frac{1}{|B(x_B,\rho(x_B))|}\int_0^{\rho(x_B)^2} \int_{B(x_B,\,\rho(x_B))} |f(x,s)|^2 dxds\right)^{1-\theta}.
\end{align}
In particular, when $\frac{\delta}{ 2} -n(1-\theta)=0$, that is, when $0< \theta=1-\frac{\delta}{2n}<1$, we have 
$$
     \mathcal I_{B,3}   \leq  C\|f\|_{\widetilde{L^2}(\mathbb R_+;{\rm BMO}_{\mathcal L})}^{2-\frac{\delta}{n}} \left( \frac{1}{|B(x_B,\rho(x_B))|} \int_0^{\rho(x_B)^2} \int_{B(x_B,\,\rho(x_B))} |f(x,s)|^2 dxds\right)^{\frac{\delta}{2n}}
$$
whenever $4r_B< \rho(x_B)$.
This estimate is useful for   Corollary \ref{thm:mainB} below.
\end{remark}

\smallskip

Observe that the proof of  Proposition~\ref{prop:Car-BMO-2}, inspired by the approach of Theorem~\ref{thm:endpoint-HE}, depends crucially on pointwise bounds for heat kernels and their time derivatives associated to $\mathcal L$, which also allow us to obtain the following result analogously.

\begin{corollary}\label{coro:Car-BMO-3}
	Suppose $V \in \mathrm{RH}_q$ for some $q>n / 2$. Let $f\in \widetilde{L^2}(\mathbb R_+;{\rm BMO}_{\mathcal L})$, we have
$$
     r_B^{-n}\int_0^{r_B^2} \int_0^{r_B^2} \int_B  \left| t\mathcal L e^{-t \mathcal L} f\,(\cdot,s)(x)\right|^2   \frac{dx dt}{t} ds \leq C \|f\|_{\widetilde{L^2}(\mathbb R_+;{\rm BMO}_{\mathcal L})}^2.
$$

Consequently,
\begin{equation}\label{eqn:inhomo-Carleson-3}
    \sup_{B=B(x_B,r_B)}  r_B^{-n}\int_0^{r_B^2} \int_B \int_0^t \left|\sqrt{t-s} \mathcal L e^{-(t-s)\mathcal L}f(\cdot, s)(x)\right|^2 ds dxdt \leq C \|f\|_{\widetilde{L^2}(\mathbb R_+;{\rm BMO}_{\mathcal L})}^2.
\end{equation}
\end{corollary}

\smallskip

\begin{proof}[The proof of Theorem~\ref{thm:mainA}]
Note that \eqref{eqn:MaxReg-I} is a consequence of  Proposition~\ref{prop:Car-BMO-2} via  H\"older's inequality.  It suffices to show \eqref{eqn:MaxReg-II}.

Fix $B=B(x_B,r_B)$.
For any $y\in \mathbb R^n$ and $s\in (0,r_B^2)$, we split $f$ as in Corollary~\ref{coro:compare-1}.
Then
$$
     r_B^{-n}\int_0^{r_B^2}\int_B
	 \big|\mathcal M_{\mathcal L}  (f)(x,t)\big|^2 dxdt \lesssim \sum_{j=1}^3  r_B^{-n}\int_0^{r_B^2}\int_B
	 \big| \mathcal M_{\mathcal L} (f_j)(x,t)\big|^2 dxdt=:\sum_{j=1}^3 \mathcal {J}_{B,j}.
$$

Note that  $\mathcal M_{\mathcal L} $  is a bounded operator on $L^2(\mathbb R_+^{n+1})$ followed from de Simon's work.
Similar to the argument for  $\mathbf J_{B,1}$ in Theorem~\ref{thm:endpoint-HE},  we have
\begin{align*}
	 \mathcal J_{B,1}\lesssim  r_B^{-n} \iint_{\mathbb R_+^{n+1}} |f_1(y,s)|^2 dyds\leq  C    \|f\|_{\widetilde{L^2}(\mathbb R_+;{\rm BMO}_{\mathcal L})}^2.
\end{align*}

Next, by \eqref{eqn:der-time-heat}, 
\begin{align*}
	\mathcal J_{B,2}
	\leq\, & r_B^{-n}\int_0^{r_B^2}\int_B  t \int_{0}^{t}  \left| \mathcal L  e^{-(t-s)\mathcal L}f_2(\cdot, s)(x)\right|^2  ds dxdt \\
	\lesssim\, &  r_B^{-n}\int_0^{r_B^2}\int_B  t \int_{0}^{t}  \left|  \frac{1}{r_B^2} \sum_{k=1}^\infty  \frac{1}{4^{2k}}\cdot \frac{1}{|4^k B|} \int_{4^k B} \left|f(y,s)-\left(f(\cdot, s)\right)_{4^k B}\right| dy\right|^2  ds dxdt \\
	\lesssim\, &  \int_0^{r_B^2} \frac{t}{r_B^4}dt\cdot   \int_0^{r_B^2}     \sum_{k=1}^\infty    \frac{1}{4^{2k}} \frac{1}{|4^k B|} \int_{4^k B} \left|f(y,s)-\left(f(\cdot, s)\right)_{4^k B}\right|^2 dyds\\
	\lesssim \, &     \|f\|_{\widetilde{L^2}(\mathbb R_+;{\rm BMO}_{\mathcal L})}^2.
\end{align*}

It remains to estimate  $\mathcal J_{B,3}$, and we also consider it in the following two cases.

\smallskip

\noindent{\it Case 1.} $4r_B\geq \rho(x_B)$. \  \
In this case, we may rewrite
$$
    f_3(y,s)=\left(f(\cdot, s)\right)_{4 B}\mathsf 1_{4B} (y)\mathsf 1_{(0,\, r_B^2)}(s)+\left(f(\cdot, s)\right)_{4 B}\mathsf 1_{(4B)^c} (y)\mathsf 1_{(0,\, r_B^2)}(s)    =:f_{3,1}+f_{3,2},
$$
and the corresponding subterms of $\mathcal J_{B,3}$ are denoted by $\mathcal J_{B,3}^{(1)}$ and $\mathcal J_{B,3}^{(2)}$.

Using the $L^2(\mathbb R_+^{n+1})$ boundedness of  $\mathcal M_{\mathcal L}$  again,
\begin{align*}
    \mathcal J_{B,3}^{(1)} \lesssim   r_B^{-n} \iint_{\mathbb R_+^{n+1}} \big|f_{3,1}(x,s)\big|^2 dxds  \lesssim    \|f\|_{\widetilde{L^2}(\mathbb R_+;{\rm BMO}_{\mathcal L})}^2.
\end{align*}

On the other hand, for any $x\in B$ and $s<t$,
\begin{align*}
	  \left|\mathcal L e^{-(t-s)\mathcal L} f_{3,2}(\cdot,s)(x)\right| \lesssim \,&  \big|(f(\cdot, s))_{4B}\big| \int_{(4B)^c}  \frac{1}{(t-s)^{(n+2)/2}} \exp\left(-\frac{|x-y|^2}{c(t-s)}\right) dy\\
	  \lesssim \,& r_B^{-2} \big|f(\cdot, s)\big|_{4B} .
\end{align*}
Therefore,
$$
    \mathcal J_{B,3}^{(2)}\lesssim \int_0^{r_B^2}  \big|f(\cdot, s)\big|_{4B} ^2\, ds \lesssim \|f\|_{\widetilde{L^2}(\mathbb R_+;{\rm BMO}_{\mathcal L})}^2.
$$

\smallskip

\noindent{\it Case 2.} $4r_B< \rho(x_B)$.  \  \  It follows from \eqref{eqn:der-time-heat-2}  that
\begin{align*}
	\mathcal J_{B,3}\lesssim \,  &    r_B^{-n}\int_0^{r_B^2}\int_B
	 \bigg(\int_0^{t}   \left(\frac{\sqrt{t-s}}{r_B}\right)^\delta  \mathbf M_B(s)     \frac{ds}{t-s}\bigg)^2 dxdt \\
	 \lesssim \,  &  \int_0^{r_B^2}    \bigg(\int_0^{t}   \left(\frac{\sqrt{t-s}}{r_B}\right)^\delta  \mathbf M_B(s)     \frac{ds}{t-s}\bigg)^2   dt \\
	 =:\, &  \int_{\mathbb R_+} \left|\int_{\mathbb R_+} K(t,s)  \left(\mathbf M_B(s)\mathsf 1_{(0,\, r_B^2)}(s)\right) ds\right|^2 dt.
\end{align*}
where
$$
   K(t,s)=\frac{(t-s)^{\frac{\delta}{2}-1} }{r_B^\delta} \mathsf 1_{(0,\,r_B^2)}(s)\cdot  \mathsf 1_{(s,\,r_B^2)}(t).
$$
It's clear that
$$
    \sup_s \int_{\mathbb R_+} |K(t,s)|dt \lesssim 1 \quad \text{and }\quad  \sup_t \int_{\mathbb R_+} |K(t,s)|ds\lesssim 1.
$$
Hence by Schur's lemma again,
\begin{align*}
	\mathcal J_{B,3}\lesssim \int_{\mathbb R_+} \left(\mathbf M_B(s)\mathsf 1_{(0, \, r_B^2)}(s)\right)^2  ds  =    \int_0^{r_B^2} \left(\mathbf M_B(s)\right)^2  ds \lesssim   \|f\|_{\widetilde{L^2}(\mathbb R_+;{\rm BMO}_{\mathcal L})}^2.
\end{align*}

From the above,  \eqref{eqn:MaxReg-II} follows readily.

\smallskip

Finally, note that
$$
    u(x,t)=\int_0^t e^{-(t-s)\mathcal L}f(\cdot, s)(x)\, ds,
$$
then
$$
   \nabla_x u(x,t)=\int_0^t  \nabla_x e^{-(t-s)\mathcal L}f(\cdot, s)(x)ds=\mathcal T_{\mathcal L}  f(x,s)
$$
and
$$
 \partial_t u(x,t)=f(x,t)+\int_0^t  \partial_t e^{-(t-s)\mathcal L}f(\cdot, s)(x)\, ds=f(x,t)+\mathcal M_{\mathcal L} (f)(x,s).
$$
Hence we complete the proof of \eqref{eqn:MaxReg-III} by combining \eqref{eqn:MaxReg-I}  and \eqref{eqn:MaxReg-II} .
\end{proof}

\medskip

\begin{remark}\label{rem:J-new}
It's clear that  $\mathcal I_{B,3}$ in Proposition~\ref{prop:Car-BMO-2} and $\mathcal J_{B,3}$ are both  bounded from above by
$$   (4r_B)^{-n}\int_0^{(4r_B)^2}   \int_{4B} |f(x,s)|^2 dxds $$ 
for the case  $4r_B\geq \rho(x_B)$.

Now consider the case $4r_B<\rho(x_B)$.  It follows from  \eqref{eqn:der-time-heat-2} that
\begin{align*}
    \mathcal J_{B,3}\lesssim\, &     r_B^{-n}\int_0^{r_B^2}\int_B
	 \bigg(\int_0^{t}   \left(\frac{\sqrt{t-s}}{r_B}\right)^{\delta\over 2}  \left(\frac{\sqrt{t-s}}{\rho(x_B)}\right)^{\delta\over 2}   \mathbf M_B(s)     \frac{ds}{t-s}\bigg)^2 dxdt \\
	 \lesssim\, &  \left( \frac{r_B}{\rho(x_B)}\right)^{\delta}   \left( r_B^{-n}\int_0^{r_B^2}\int_B
	 \bigg(\int_0^{t}   \left(\frac{\sqrt{t-s}}{r_B}\right)^{\delta}   \mathbf M_B(s)     \frac{ds}{t-s}\bigg)^2 dxdt \right) .
\end{align*}
Then by the same argument above for proving  \eqref{eqn:MaxReg-II}, we have a certain sharper estimate that
$$
    \mathcal J_{B,3}\leq   C    \left(\frac{r_B}{\rho(x_B)}\right)^{\delta}     \int_0^{r_B^2}\left( \mathbf M_B(s)\right)^2  ds. 
$$ 
Similarly, using the first inequality in \eqref{eqn:der-time-heat-2}, we can also obtain
$$
    \mathcal J_{B,3} \leq C  \left(\frac{r_B}{\rho(x_B)}\right)^{2\delta } \int_0^{r_B^2}   \big|f(\cdot,s)\big|_{4B}^2 \, ds.
$$

Hence, $\mathcal J_{B,3}$ and $\mathcal I_{B,3}$  have  similar expressions of upper bounds  whenever $4r_B<\rho(x_B)$; see Remark~\ref{rem:JB3} for details.
\end{remark}

This remark is useful for further discussions on maximal regularity in $\widetilde{L^2}(\mathbb R_+; {\rm CMO}_{\mathcal L})$, which can be defined in terms of limiting behavior conditions, similar to that of ${\rm CMO}_{\mathcal L}$ in \cite{SW2022}.
Indeed, our argument for Theorem \ref{thm:mainA}, Remarks \ref{rem:JB3} and \ref{rem:J-new}, allows us to obtain the following local version of maximal regularity in ${\rm CMO}_{\mathcal L}$.

\begin{corollary}\label{thm:mainB}
Suppose $V \in \mathrm{RH}_q$ for some $q>n / 2$. For every $f\in  \widetilde{L^2}(\mathbb R_+;{\rm CMO}_{\mathcal L}(\mathbb R^n))$,  the estimates \eqref{eqn:MaxReg-II}, \eqref{eqn:MaxReg-I} and \eqref{eqn:MaxReg-III}  in Theorem \ref{thm:mainA} also hold. Additionally,
$$
	 \lim _{a \rightarrow 0}\sup _{B: r_{B} \leq a}  \,\mathcal{C}(f)_{B,j}
     =  \lim _{a \rightarrow \infty}\sup _{B: r_{B} \geq a}  \,\mathcal{C}(f)_{B,j}
     =   \lim _{a \rightarrow \infty}\sup _{B: B \subseteq \left(B(0, a)\right)^c}  \,\mathcal{C}(f)_{B,j}=0,\quad j=1,2,
$$
where
$$
 \mathcal{C}(f)_{B,1}=\bigg( r_B^{-n} \int_0^{r_B^2}\int_B
	 \big|\mathcal M_{\mathcal L}  (f)(x,t)\big|^2 dxdt\bigg)^{1/2}  
  \  \    {\rm and}\  \  
    \mathcal{C}(f)_{B,2}=\left( r_B^{-n} \int_0^{r_B^2}\int_B \big|\mathcal T_{\mathcal L}  (f)(x,t)\big|^2 \frac{dxdt}{t}\right)^{1/2}.
$$
\end{corollary}

\medskip


\section{Maximal regularity in ${\rm BMO}_{\mathcal L}$: a global version} \label{sec:global-BMOL}
\setcounter{equation}{0}

Eventually, based on the local estimates in the previous section,
now we will prove our second main result, namely Theorem \ref{thm:improved-mainA},  the analog associated to $\mathcal L$ of our Theorem \ref{thm:endpoint-HE}.

Observe that there is no mean oscillation structure when establishing our previous (locally) Carleson-type estimates in Theorem \ref{thm:mainA}. Our argumnet ensures that it's not necessary when the time interval is restricted to the scale compared to $r_B^2$.
When transferring to Theorem \ref{thm:improved-mainA}, the mean oscillation occurring in the left-hand side of  \eqref{eqn:MaxReg-II-classical} hints us to apply the argument of Theorem \ref{thm:endpoint-HE} to obtain an additional derivative, which will improve corresponding decay estimates by exploiting elaborate heat kernel analysis.

\begin{proof}[The proof of Theorem \ref{thm:improved-mainA}]

We consider \eqref{eqn:global-MaxReg-II} in the following two cases: $4r_B<\rho(x_B)$, $4r_B\geq \rho(x_B)$.

\smallskip

\noindent{\it Case 1. $4r_B<\rho(x_B)$}. In this case, we need to prove that 
$$
       r_B^{-n}\int_{\mathbb R_+}  \int_B  \left |\mathcal M_{\mathcal L}  (f) (x,t) -\left(\mathcal M_{\mathcal L}  (f) (\cdot,t)\right)_B \right|^2 dx dt
       \lesssim   \|f\|_{{\mathbb L}^2(\mathbb R_+; {\rm BMO}_{\mathcal L})}^2
$$
with the implicit constant independent of the ball $B$. 

We split $f=f_1+f_2+f_3$ as in the  Proof of Theorem~\ref{thm:endpoint-HE}, and let
$$
    \mathtt J_{B,j} :=   r_B^{-n}\int_{\mathbb R_+}  \int_B  \left |\mathcal M_{\mathcal L}  (f_j) (x,t) -\left(\mathcal M_{\mathcal L} ( f_j )(\cdot,t)\right)_B \right|^2 dx dt,\quad j=1,2,3.
$$

By  the $L^2(\mathbb R_+^{n+1})$ boundedness of $\mathcal M_{\mathcal L} $,
\begin{align*}
	 \mathtt J_{B,1}\lesssim  r_B^{-n} \iint_{\mathbb R_+^{n+1}} |f_1(y,s)|^2 dyds\leq    \|f\|_{{\mathbb L}^2(\mathbb R_+;{\rm BMO}_{\mathcal L})}^2.
\end{align*}

We are now starting to consider  $\mathtt J_{B,2}$.  Note that
$$
    \mathtt J_{B,2} :=   r_B^{-n}\int_{\mathbb R_+}  \int_B  \bigg|  \int_0^t  \left(\mathcal Le^{-(t-s)\mathcal L}f_2(\cdot, s)(x)-  \big(\mathcal Le^{-(t-s)\mathcal L}f_2(\cdot, s) \big)_B\right)ds \bigg|^2 dxdt,
$$
then same as the decomposition of the term $ \mathbf J_{B,2}$ in the proof of Theorem \ref{thm:endpoint-HE}, we rewrite
\begin{align*}
      \mathtt J_{B,2} &\lesssim r_B^{-n}  \int_{0}^{(8r_B)^2} \int_B  \bigg |  \int_0^t   \left(\mathcal Le^{-(t-s)\mathcal L}f_2(\cdot, s)(x)-  \big(\mathcal Le^{-(t-s)\mathcal L}f_2(\cdot, s) \big)_B\right) ds \bigg| ^2  dxdt \\
      &\quad +   r_B^{-n}  \int_{(8r_B)^2}^{\infty} \int_B  \bigg |  \int_{t-(4r_B)^2}^t \left( \mathcal Le^{-(t-s)\mathcal L}f_2(\cdot, s)(x)-  \big(\mathcal Le^{-(t-s)\mathcal L}f_2(\cdot, s) \big)_B \right) ds\bigg|^2 dxdt \\
          &\quad +   r_B^{-n}  \int_{(8r_B)^2}^{\infty} \int_B  \bigg |  \int_0^{t-(4r_B)^2} \left( \mathcal Le^{-(t-s)\mathcal L}f_2(\cdot, s)(x)-  \big(\mathcal Le^{-(t-s)\mathcal L}f_2(\cdot, s)\big)_B \right) ds\bigg|^2 dxdt \\
     &=:  \mathtt J_{B,2}^{(1)} +\mathtt J_{B,2}^{(2,1)}+\mathtt J_{B,2}^{(2,2)}.
\end{align*}

Note that
\begin{align*}
     \mathtt J_{B,2}^{(1)} \lesssim r_B^{-n}  \int_{0}^{(8r_B)^2} \int_B  \left |  \int_0^t    \mathcal Le^{-(t-s)\mathcal L}f_2(\cdot, s)(x) ds \right| ^2  dxdt ,
\end{align*}
then the argument of $\mathcal J_{B,2}$ in the proof of Theorem \ref{thm:mainA} yields that 
$$
     \mathtt J_{B,2}^{(1)}  \lesssim  \|f\|_{{\mathbb L}^2(\mathbb R_+;{\rm BMO}_{\mathcal L})}^2.
$$

Combining H\"older's inequality and Poincar\'e's inequality,
\begin{align*}
    \mathtt J_{B,2}^{(2,1)}  &= r_B^{-n+2} \int_{(8r_B)^2-(4r_B)^2 }^\infty \int_{\max\{(8r_B)^2,s\}}^{s+(4r_B)^2}    \int_B \left |   \mathcal Le^{-(t-s)\mathcal L}f_2(\cdot, s)(x)-  \big(\mathcal Le^{-(t-s)\mathcal L}f_2(\cdot, s)\big)_B \right|^2 dx  dtds\\
    &\lesssim r_B^{-n+4}  \int_{(4r_B)^2 }^\infty \int_{0}^{(4r_B)^2}     \int_B \left | \nabla_x e^{-\tau \mathcal L} \mathcal L f_2(\cdot, s)(x) \right|^2 dx  d\tau ds.
\end{align*}

Combining \eqref{eqn:space-der-1} in  Proposition~\ref{prop:space-derivative}, we have
\begin{align*}
       \mathtt J_{B,2}^{(2,1)} &\lesssim   r_B^{-n+4}  \int_{(4r_B)^2 }^\infty   \int_{0}^{(8r_B)^2}     \int_{2B}
    \Big ( \left|\mathcal L^2 e^{-\tau \mathcal L} f_2(\cdot, s)(x)\right|  \left |\mathcal Le^{-\tau \mathcal L} f_2(\cdot,s )(x)\right| +\frac{1}{r_B^2} \left|\mathcal Le^{-\tau \mathcal L} f_2(\cdot, s)(x)\right|^2 \Big )\,  dxd\tau  ds.
\end{align*}

Similar to the estimate on $\mathcal I_{B,2}$ in Theorem~\ref{thm:mainA}, one may apply Lemma~\ref{lem:heat-Schrodinger} to  obtain an analogous of \eqref{eqn:I2-aux-new} that for any $x\in 2B$,  $0<\tau<(8r_B)^2$ and for any fixed $s$,
\begin{align*}
&\left|\mathcal L^2 e^{-\tau \mathcal L} f_2(\cdot, s)(x)\right|  \left |\mathcal Le^{-\tau \mathcal L} f_2(\cdot,s )(x)\right| +\frac{1}{r_B^2} \left|\mathcal Le^{-\tau \mathcal L} f_2(\cdot, s)(x)\right|^2 \\
 \lesssim\,&  \frac{1}{r_B^6}      \sum_{k=1}^\infty  \frac{1}{4^{2k}}\cdot \frac{1}{|4^k B|} \int_{4^k B} \left|f(y,s)-\left(f(\cdot, s)\right)_{4^k B}\right|^2 dy.
\end{align*}
 This yields 
$$
     \mathtt J_{B,2}^{(2,1)}  \lesssim  \|f\|_{{\mathbb L}^2(\mathbb R_+;{\rm BMO}_{\mathcal L})}^2.
$$

For the term $\mathtt J_{B,2}^{(2,2)}$,  note that for  any $t>(8r_B)^2$, $0<s< t-(4r_B)^2$ and $x\in B$, then $t-s>(4r_B)^2$,  and thus one may apply (iii) of Lemma~\ref{lem:heat-Schrodinger} to obtain that for every $0<\delta<\min\left\{1,2-\frac{n}{q}\right\}$,
\begin{align*}
     &\left|\mathcal Le^{-(t-s)\mathcal L}f_2(\cdot, s)(x)-  \big(\mathcal Le^{-(t-s)\mathcal L}f_2(\cdot, s)\big)_B\right|\\
     \leq\, &\frac{1}{|B|}\int_B   \left|\mathcal Le^{-(t-s)\mathcal L}f_2(\cdot, s)(x)-  \mathcal Le^{-(t-s)\mathcal L}f_2(\cdot, s)(y)\right|dy\\
     \lesssim \, &\frac{1}{|B|}   \int_B    \int_{(4B)^c}      \left(\frac{r_B}{\sqrt{t-s}}\right)^\delta    \frac{1}{(t-s)^{\frac{n+2}{2}}} \exp\left(-\frac{|x-z|^2}{c(t-s)}\right)   \left| f_2(z,s)\right| dzdy\\
     \lesssim \, &    \left(\frac{r_B}{\sqrt{t-s}}\right)^\delta       \int_{(4B)^c}     \frac{1}{(t-s)^{\frac{\theta}{2}}}  \frac{1}{|x-z|^{n+2-\theta}}   \left| f_2(z,s)\right| dz \\
     \lesssim\, &  \left(\frac{r_B}{\sqrt{t-s}}\right)^\delta\frac{1}{(t-s)^{\frac{\theta}{2}}}        \sum_{j=1}^\infty \int_{4^{j+1}B\setminus 4^j B}        \frac{1}{|x-z|^{n+2-\theta}}   \bigg[ \left|f(z,s)-\left(f(\cdot, s)\right)_{4^{j+1}B}\right|\bigg.\\
     &\qquad \qquad \qquad \qquad \qquad
      +\bigg.\sum_{k=1}^j  \frac{1}{|4^{k+1}B|} \int_{4^{k+1}B} \left|f(u,s)-\left(f(\cdot, s)\right)_{4^{k+1}B}\right| du \bigg ] \,dz\\
     \lesssim\, &   \frac{r_B^ {\delta+\theta-2}  }   {  (t-s)^ {\frac{\delta+\theta}{2} } }
           \sum_{k=1}^\infty  \frac{1}{4^{(2-\theta)k}}\cdot \frac{1}{|4^k B|} \int_{4^k B} \left|f(u,s)-\left(f(\cdot, s)\right)_{4^k B}\right|^2 du
\end{align*}
holds for any $0<\theta <2$.  Now we fix some $\delta<\min\left\{1,2-\frac{n}{q}\right\}$ and then choose $0<\theta<2$ such that $\delta+\theta >2$, which is feasible even for the case $V\in {\rm RH}_{n/2}$, provided by the  self-improving property of the ${\rm RH}_q$ class: there exists $\varepsilon>0$
depending only on $n$ and the constant $C$ in \eqref{eqn:Reverse-Holder}, such that $V\in {\rm RH}_{q+\varepsilon}$; see \cite{G}. 

Denote
$$
    {\mathtt { Avg}}_B(s)=   \sum_{k=1}^\infty  \frac{1}{4^{(2-\theta)k}}\cdot \frac{1}{|4^k B|} \int_{4^k B} \left|f(y,s)-\left(f(\cdot, s)\right)_{4^k B}\right|^2 dy,
$$
we have 
\begin{align*}
   \mathtt J_{B,2}^{(2,2)}&\lesssim   r_B^ {2\cdot(\delta+\theta-2)}  \int_{(8r_B)^2}^{\infty}  \bigg| \int_0^{t-(4r_B)^2}  \frac{ 1}   {  (t-s)^ {\frac{\delta+\theta}{2} } }   {\mathtt { Avg}}_B(s)\, ds \bigg|^2 dt\\
   &=   r_B^ {2\cdot(\delta+\theta-2)}  \int_{\mathbb R_+}  \bigg| \int_{\mathbb R_+}   {\mathtt K}(t,s)  {\mathtt { Avg}}_B(s)\, ds \bigg|^2 dt,
\end{align*}
where 
$$
  {\mathtt K}(t,s)=\frac{ 1}   {  (t-s)^ {\frac{\delta+\theta}{2} } } \mathsf 1_{((8r_B)^2,\,+\infty)}(t)\cdot \mathsf 1_{(0,\,t-(4r_B)^2)}(s).
$$
Since $\delta+\theta>2$, then
$$
      \sup_s \int_{\mathbb R_+}  {\mathtt K}(t,s)\, dt\lesssim r_B^{2-(\delta+\theta)}\quad {\rm and}\quad \sup_t \int_{\mathbb R_+}  {\mathtt K}(t,s) \, ds\lesssim r_B^{2-(\delta+\theta)}.
$$
By Schur's lemma again, we have
$$
     \mathtt J_{B,2}^{(2,2)} \lesssim   \int_{\mathbb R_+}  \big(   {\mathtt { Avg}}_B(s) \big)^2 ds\lesssim \|f\|_{{\mathbb L}^2(\mathbb R_+;{\rm BMO}_{\mathcal L})}^2.
$$

\smallskip

It remains to estimate $\mathtt J_{B,3}$. Note that $e^{-t\mathcal L}1\neq 1$, there is no conservation law here, and one may combine arguments in the proof of Theorem \ref{thm:mainA} and some new observations, based on different scales of $t-s$. 

One may further assume that $8r_B<\rho(x_B)$ (the proof for the other case $\rho(x_B)\leq 8r_B$ is similar and simpler, we skip it), then we can  rewrite
\begin{align*}
    \mathtt J_{B,3} 
    \lesssim\, &   r_B^{-n}\int_{0}^{\rho(x_B)^2}    \int_B  \bigg |\int_0^t \left(  \mathcal Le^{-(t-s)\mathcal L}f_3(\cdot, s)(x) - \big(\mathcal Le^{-(t-s)\mathcal L}f_3(\cdot, s)\big)_B\right) ds  \bigg|^2 dx dt\\
    &+ r_B^{-n}\int_{\rho(x_B)^2}^\infty    \int_B  \bigg |  \int_{t-(4r_B)^2}^t     \left(  \mathcal Le^{-(t-s)\mathcal L}f_3(\cdot, s)(x) - \big(\mathcal Le^{-(t-s)\mathcal L}f_3(\cdot, s)\big)_B\right) ds  \bigg|^2 dx dt\\
    &+ r_B^{-n}\int_{\rho(x_B)^2}^\infty    \int_B  \bigg | \int_{t-\rho(x_B)^2}^{t-(4r_B)^2}  \left(  \mathcal Le^{-(t-s)\mathcal L}f_3(\cdot, s)(x) - \big(\mathcal Le^{-(t-s)\mathcal L}f_3(\cdot, s)\big)_B\right) ds  \bigg|^2 dx dt\\
    &+ r_B^{-n}\int_{\rho(x_B)^2}^\infty    \int_B  \bigg |\int_0^{t-\rho(x_B)^2} \left(  \mathcal Le^{-(t-s)\mathcal L}f_3(\cdot, s)(x) - \big(\mathcal Le^{-(t-s)\mathcal L}f_3(\cdot, s)\big)_B\right) ds  \bigg|^2 dx dt\\
  =:\, &   \mathtt J_{B,3} ^{(1)}+    \mathtt J_{B,3} ^{(2,1)}+ \mathtt J_{B,3} ^{(2,2)}+ \mathtt J_{B,3} ^{(2,3)}.
\end{align*}

Observe that 
$$
    \mathtt J_{B,3} ^{(2,1)}\lesssim    r_B^{-n}\int_{\rho(x_B)^2}^\infty    \int_B \bigg|  \int_{t-(4r_B)^2}^t   \mathcal Le^{-(t-s)\mathcal L}f_3(\cdot, s)(x) \, ds   \bigg|^2 dxdt.
$$
As shown in \eqref{eqn:der-time-heat-2} that 
$$
  \left|\partial_t e^{-(t-s) \mathcal L}f_3(\cdot, s)(x)\right|\leq C 
   \frac{1}{t-s}\left(\frac{\sqrt{t-s}}{r_B}\right)^{\delta}   \mathbf M_B(s) 
$$
for each $x\in B$,
where $\mathbf M_B(s)$ is given in \eqref{MBs}.  Note that $0<t-s<(4r_B)^2$ in the expression of  $\mathtt J_{B,3} ^{(2,1)}$, then  by  the argument for $\mathcal J_{B,3}$ (in the proof Theorem \ref{thm:mainA} for  the case $4r_B<\rho(x_B)$),  one may verify that   $\mathtt J_{B,3} ^{(2,1)}$ is  bounded by 
$\displaystyle  \int_{\mathbb R_+}   \big(  {\mathbf M}_B(s)  \big) ^2   ds  $. 
Therefore,
$$
   \mathtt J_{B,3} ^{(2,1)} \lesssim   \|f\|_{{\mathbb L}^2(\mathbb R_+;{\rm BMO}_{\mathcal L})}^2.
$$

Let's consider the terms $\mathtt J_{B,3} ^{(1)}$ and $\mathtt J_{B,3} ^{(2,2)}$. Note that  $0<t-s<\rho(x_B)^2$ therein and this allows us to estimate them in the same manner. For each $x\in B$, by \eqref{eqn:der-time-heat-2},
$$
\left|  \mathcal Le^{-(t-s)\mathcal L}f_3(\cdot, s)(x) - \big(\mathcal Le^{-(t-s)\mathcal L}f_3(\cdot, s)\big)_B\right| \lesssim \frac{1}{t-s}\left(\frac{\sqrt{t-s}}{\rho(x_B)}\right)^{\delta} \big \lfloor \log_2 \frac{\rho(x_B)}{r_B} \big\rfloor\cdot  \mathbf M_B(s) .
$$
On the other hand,  apply (iii) of Lemma~\ref{lem:heat-Schrodinger}  again to see

\begin{align*}
     &\left|\mathcal Le^{-(t-s)\mathcal L}f_3(\cdot, s)(x)-  \big(\mathcal Le^{-(t-s)\mathcal L}f_3(\cdot, s)\big)_B\right|\\
     \leq\, &   \big |f(\cdot, s)_{4B}\big|   \frac{1}{|B|}\int_B   \left|\mathcal Le^{-(t-s)\mathcal L}(1)(x)-  \mathcal Le^{-(t-s)\mathcal L} (1)(y)\right|dy\\
     \lesssim \, &  \big \lfloor \log_2 \frac{\rho(x_B)}{r_B} \big\rfloor     \cdot \mathbf M_B(s)     \int_{\mathbb R^n}      \left(\frac{r_B}{\sqrt{t-s}}\right)^\delta    \frac{1}{(t-s)^{\frac{n+2}{2}}} \exp\left(-\frac{|x-z|^2}{c(t-s)}\right)     dz\\
     \lesssim \, &  \frac{1}{t-s}\left(\frac{r_B}{\sqrt{t-s}}\right)^{\delta} \big \lfloor \log_2 \frac{\rho(x_B)}{r_B} \big\rfloor     \cdot \mathbf M_B(s)  .
\end{align*}
Combining them, we have 
\begin{align}\label{eqn:aux-JB3-new}
     \left|\mathcal Le^{-(t-s)\mathcal L}f_3(\cdot, s)(x)-  \big(\mathcal Le^{-(t-s)\mathcal L}f_3(\cdot, s)\big)_B\right|
     \lesssim \, &  \frac{1}{t-s}\left(\frac{\sqrt{t-s}}{\rho(x_B)}\right)^{\frac{2\delta}{3}} 
     \left(\frac{r_B}{\sqrt{t-s}}\right)^{\delta\over 3} \big \lfloor \log_2 \frac{\rho(x_B)}{r_B} \big\rfloor     \cdot \mathbf M_B(s)  \notag\\
      \lesssim \, &    \left(\frac{\sqrt{t-s}}{\rho(x_B)}\right)^{\delta\over 3}   \frac{1}{t-s}\mathbf M_B(s) .
\end{align}
and also
\begin{align*}
     \left|\mathcal Le^{-(t-s)\mathcal L}f_3(\cdot, s)(x)-  \big(\mathcal Le^{-(t-s)\mathcal L}f_3(\cdot, s)\big)_B\right|
     \lesssim \, &  \frac{1}{t-s}\left(\frac{r_B}{\rho(x_B)}\right)^{\delta\over 2} \big \lfloor \log_2 \frac{\rho(x_B)}{r_B} \big\rfloor     \cdot \mathbf M_B(s) \\
      \lesssim \, &  \left(\frac{r_B}{\rho(x_B)}\right)^{\delta\over 4} \frac{1}{t-s}\mathbf M_B(s) .
\end{align*}
Then
\begin{align*}
   \mathtt J_{B,3} ^{(1)} &\lesssim  \int_{0}^{\rho(x_B)^2}     \bigg|  \int_{0}^{t}  \left(\frac{\sqrt{t-s}}{\rho(x_B)}\right)^{\delta\over 3}   \frac{1}{t-s} {\mathbf M}_B(s)   \,ds  \bigg|^2  dt
\end{align*} 
and
\begin{align*}
   \mathtt J_{B,3} ^{(2,2)} &\lesssim  \left(\frac{r_B}{\rho(x_B)}\right)^{\delta\over 2} \int_{\rho(x_B)^2}^\infty     \bigg|  \int_{t-\rho(x_B)^2}^{t-(4r_B)^2}    \frac{1}{t-s} \mathbf M_B(s)   \, ds  \bigg|^2 dt.
\end{align*}
One may verify that by Schur's lemma, the operator ${T^{(1)}}({\mathbf {M}}_B)(t):=\int_{\mathbb R_+}    K^{(1)}(t,s) {\mathbf {M}}_B(s)\, ds$ with
$$ 
\quad  K^{(1)}(t,s)= \left(\frac{\sqrt{t-s}}{\rho(x_B)}\right)^{\delta\over 3}   \frac{1}{t-s} {\mathsf 1_{(0,\,\rho(x_B)^2)}(t)\cdot \mathsf 1_{ (0,\,t ) }(s)}
$$
is bounded from $L^2(\mathbb R_+)$ to $L^2(\mathbb R_+)$ with  norm $\displaystyle \left\|T^{(1)}\right\|_{L^2(\mathbb R_+)\to L^2(\mathbb R_+)}\lesssim 1 $. Therefore,
$$
   \mathtt J_{B,3} ^{(1)} \lesssim   \|f\|_{{\mathbb L}^2(\mathbb R_+;{\rm BMO}_{\mathcal L})}^2.
$$

Similarly, $T^{(2,2)}(\mathbf M_B)(t):=\int_{\mathbb R_+}    K^{(2,2)}(t,s) \mathbf M_B(s)\, ds $ with
$$
\quad  K^{(2,2)}(t,s)=\frac{1}{t-s} \mathsf 1_{(\rho(x_B)^2,\,+\infty)}(t)\cdot \mathsf 1_{ ( t-\rho(x_B)^2, \,t-(4r_B)^2 ) }(s)
$$
is bounded from $L^2(\mathbb R_+)$ to $L^2(\mathbb R_+)$ with  norm $\displaystyle \left\|T^{(2,2)} \right\|_{L^2(\mathbb R_+)\to L^2(\mathbb R_+)}\lesssim \ln \frac{\rho(x_B)}{r_B} $.  Then
\begin{align*}
   \mathtt J_{B,3} ^{(2,2)} &\lesssim  \left(\frac{r_B}{\rho(x_B)}\right)^{\delta\over 2}   \left(\ln \frac{\rho(x_B)}{r_B} \right)^2 \int_{\mathbb R_+} \left(\mathbf M_B(s)\right)^2   ds\lesssim \int_{\mathbb R_+} \left(\mathbf M_B(s)\right)^2   ds \lesssim   \|f\|_{{\mathbb L}^2(\mathbb R_+;{\rm BMO}_{\mathcal L})}^2.
\end{align*}

In the end we estimate $\mathtt J_{B,3} ^{(2,3)} $. Note that $t-s>\rho(x_B)$ here, thus combining Lemma \ref{lem:heat-Schrodinger} and the fact $\rho(x)\approx \rho(x_B)$ for $x\in B$ by Lemma \ref{lem:size-rho}, we can 
imporve the estimate \eqref{eqn:aux-JB3-new} to 
\begin{align*}
  \left|\mathcal Le^{-(t-s)\mathcal L}f_3(\cdot, s)(x)-  \big(\mathcal Le^{-(t-s)\mathcal L}f_3(\cdot, s)\big)_B\right|
     \lesssim \,  \frac{1}{t-s} \left(\frac{\sqrt{t-s}}{\rho(x_B)}\right)^{-N}  \mathbf M_B(s)
\end{align*}
for any $N\in \mathbb N_+$.
Hence it follows from Schur's lemma to obtain
$$
     \mathtt J_{B,3} ^{(2,3)} \lesssim   \|f\|_{{\mathbb L}^2(\mathbb R_+;{\rm BMO}_{\mathcal L})}^2.
$$

\smallskip

\noindent{\it Case 2. $4r_B\geq \rho(x_B)$}. In this case, we may further assume that $r_B\geq \rho(x_B)$ and it needs to prove that 
\begin{equation}\label{eqn:global-case2}
       r_B^{-n}\int_{\mathbb R_+}  \int_B  \left |\mathcal M_{\mathcal L}  f (x,t) \right|^2 dx dt
       \lesssim   \|f\|_{{\mathbb L}^2(\mathbb R_+; {\rm BMO}_{\mathcal L})}^2.
\end{equation}
Indeed,  if $4r_B\geq \rho(x_B)$ while $r_B< \rho(x_B)$, then 
\begin{align*}
       r_B^{-n}\int_{\mathbb R_+}  \int_B  \left |\mathcal M_{\mathcal L}  f (x,t) -\left(\mathcal M_{\mathcal L}  f (\cdot,t)\right)_B \right|^2 dx dt &\lesssim   r_B^{-n}\int_{\mathbb R_+}  \int_B  \left |\mathcal M_{\mathcal L}  f (x,t)   \right|^2 dx dt \\
        &\lesssim   (4r_B)^{-n}\int_{\mathbb R_+}  \int_{4B}  \left |\mathcal M_{\mathcal L}  f (x,t)   \right|^2 dx dt ,
\end{align*}
it also shifts to prove \eqref{eqn:global-case2} whenever $r_B\geq \rho(x_B)$.

Now, we split $f$ as 
$$
  f(\cdot, s)= f(\cdot, s)\mathsf 1_{4B}+  f(\cdot, s)\mathsf 1_{(4B)^c}=: \tilde{f}_1(\cdot, s)+\tilde{f}_2(\cdot, s),
$$
and let
$$
    \mathfrak J_{B,j} :=   r_B^{-n}\int_{\mathbb R_+}  \int_B  \left |\mathcal M_{\mathcal L}  \tilde{f}_j (x,t)  \right|^2 dx dt,\quad j=1,2.
$$

Apply the $L^2(\mathbb R_+^{n+1})$ boundedness of $\mathcal M_{\mathcal L} $ again,
$$
    \mathfrak J_{B,1} \lesssim r_B^{-n}  \left\|\tilde{f}_1\right\|_{L^2(\mathbb R_+^{n+1})}  \lesssim   \|f\|_{{\mathbb L}^2(\mathbb R_+;{\rm BMO}_{\mathcal L})}^2.
$$

Let's consider the term $ \mathfrak J_{B,2}$. Rewrite
\begin{align*}
    \mathfrak J_{B,2}& \lesssim  r_B^{-n}\int_{0}^{(8r_B)^2} \int_B \bigg|  \int_0^t \mathcal L e^{-(t-s)\mathcal L}  \tilde{f}_2(\cdot, s)(x)\, ds\bigg|^2 dxdt\\
  &\quad +  r_B^{-n}\int_{(8r_B)^2}^{\infty} \int_B \left| \int_{t-(4r_B)^2}^t   \mathcal L e^{-(t-s)\mathcal L}  \tilde{f}_2(\cdot, s)(x)\, ds\right|^2 dxdt\\
  &\quad +  r_B^{-n}\int_{(8r_B)^2}^{\infty} \int_B \left|  \int_0^{t-(4r_B)^2}  \mathcal L e^{-(t-s)\mathcal L}  \tilde{f}_2(\cdot, s)(x)\, ds\right|^2 dxdt\\
  &=:   \mathfrak J_{B,2}^{(1)}+ \mathfrak J_{B,2}^{(2,1)}+ \mathfrak J_{B,2}^{(2,2)}.
\end{align*}

Note that $t-s\lesssim r_B^2$ in the terms $ \mathfrak J_{B,2}^{(1)}$ and $ \mathfrak J_{B,2}^{(2,1)}$, one may apply H\"older's inequality directly to obtain
\begin{align*}
    \mathfrak J_{B,2}^{(1)}+  \mathfrak J_{B,2}^{(2,1)} &\lesssim r_B^{-n+2}\int_{0}^{(8r_B)^2} \int_B  \int_0^t \left| \mathcal L e^{-(t-s)\mathcal L}  \tilde{f}_2(\cdot, s)(x)\right|^2  ds dxdt\\
    &\quad + r_B^{-n+2}\int_{(8r_B)^2}^{\infty} \int_B  \int_{t-(4r_B)^2}^t \left| \mathcal L e^{-(t-s)\mathcal L}  \tilde{f}_2(\cdot, s)(x)\right|^2  ds dxdt\\
    &\lesssim r_B^{-n+2}\int_{\mathbb R_+} \int_B  \int_0^{(8r_B)^2} \left| \mathcal L e^{-\tau \mathcal L}  \tilde{f}_2(\cdot, s)(x)\right|^2   dx\tau ds.
\end{align*}

Since for each $x\in B$,
\begin{align*}
\left| \mathcal L e^{-\tau \mathcal L}  \tilde{f}_2(\cdot, s)(x)\right|&\lesssim \sum_{k=1}^\infty \int_{4^{k+1}B\setminus {4^k B}} \frac{1}{|x-y|^{n+2}} |f(y, s)|\,dy\\
& \lesssim  \frac{1}{r_B^2}      \sum_{k=1}^\infty \frac{1}{4^{2k}}     \frac{1}{|4^{k+1}B|}\int_{4^{k+1}B}  |f(y, s)| \, dy,
\end{align*}
then 
$$
    \mathfrak J_{B,2}^{(1)}+  \mathfrak J_{B,2}^{(2,1)} \lesssim \int_{\mathbb R_+}  \sum_{k=1}^\infty \frac{1}{4^{2k}}   \frac{1}{|4^{k+1}B|}\int_{4^{k+1}B}  |f(y, s)|^2 \, dy ds  \lesssim   \|f\|_{{\mathbb L}^2(\mathbb R_+;{\rm BMO}_{\mathcal L})}^2.
$$

It suffices to consider $\mathfrak J_{B,2}^{(2,2)} $.  By Lemmas  \ref{lem:heat-Schrodinger} and  \ref{lem:size-rho},  for each $x\in B$ and 
$0<s<t-(4r_B)^2$,  we have for the case $4r_B\geq \rho(x_B)$,
\begin{align*}
    \left| \mathcal L e^{-(t-s)\mathcal L}  \tilde{f}_2(\cdot, s)(x)\right| 
& \lesssim  \frac{1}{r_B}\frac{1}{\sqrt{t-s}}   \left(\frac{\sqrt{t-s}}{\rho(x)}\right)^{-N}   \sum_{k=1}^\infty \frac{1}{4^{k}}     \frac{1}{|4^{k+1}B|}\int_{4^{k+1}B}  |f(y, s)| \, dy\\
& \lesssim  \frac{1}{r_B}\frac{1}{(t-s)^{\frac{N+1}{2}}}   \left(\frac{r_B}{\rho(x_B)}\right)^{\frac{k_0 }{k_0+1}N} \rho(x_B)^N   \sum_{k=1}^\infty \frac{1}{4^{k}}     \frac{1}{|4^{k+1}B|}\int_{4^{k+1}B}  |f(y, s)| \, dy
 \end{align*}
for any $N\in \mathbb N_+$. Thus
\begin{align*}
     \mathfrak J_{B,2}^{(2,2)}\lesssim \int_{(8r_B)^2}^{\infty} \left|  \int_{\mathbb R_+} \mathfrak K(t,s) \tilde{\rm Avg}(s)\,ds  \right|^2 dt,
\end{align*}
where 
$$
   \mathfrak K(t,s)=  \frac{r_B^{\frac{k_0}{k_0+1}N} }{r_B} \frac{1}{(t-s)^{\frac{N+1}{2}}}     \rho(x_B)^{\frac{1}{k_0+1}N}   \mathsf 1_{((8r_B)^2,\,+\infty)}(t) \cdot\mathsf 1_{(0, \,t-(4r_B)^2)}(s)   
$$
and 
$$
   \tilde{\rm Avg}(s)=\sum_{k=1}^\infty \frac{1}{4^{k}}     \frac{1}{|4^{k+1}B|}\int_{4^{k+1}B}  |f(y, s)| \, dy.
$$
Fix $N\geq 2$ and note that $4r_B\geq \rho(x_B)$,
$$
   \sup_s \int_{\mathbb R_+}  \mathfrak K(t,s)\, dt =\frac{r_B^{\frac{k_0}{k_0+1}N} }{r_B}    \rho(x_B)^{\frac{1}{k_0+1}N}    \int_{(4r_B)^2 }^\infty  \frac{1}{\tau ^{\frac{N+1}{2}}}  d\tau\lesssim \left(\frac{\rho(x_B)}{r_B}\right)^{\frac{1}{k_0+1}N}\lesssim 1.
$$
and also $  {\displaystyle \sup_t} \int_{\mathbb R_+}  \mathfrak K(t,s)\, ds\lesssim 1$.
Hence the operator $T$ given by $ T(\tilde{\rm Avg})(t)=\int_{\mathbb R_+} \mathfrak K(t,s) \tilde{\rm Avg}(s)\,ds$ is bounded from $L^2(\mathbb R_+)$ to $L^2(\mathbb R_+)$ with norm $\|T\|_{L^2(\mathbb R_+)\to L^2(\mathbb R_+)}\lesssim 1$, and hence
$$
     \mathfrak J_{B,2}^{(2,2)} \lesssim \int_{\mathbb R_+}  \left(  \tilde{\rm Avg}(s) \right)^2 ds   
     \lesssim   \|f\|_{{\mathbb L}^2(\mathbb R_+;{\rm BMO}_{\mathcal L})}^2.
$$

From the above, we complete the proof of \eqref{eqn:global-MaxReg-II}, and \eqref{eqn:global-MaxReg-III} is a straightforward consequence.
\end{proof}

\smallskip

\begin{remark}\label{rem:global-TL}
It's interesting to consider whether or not we have the global maximal regularity 
$$
   	\big\|  t^{-1/2}\mathcal T_{\mathcal L} f \big\|_{{\mathbb L}^2(\mathbb R_+; {\rm BMO}_{\mathcal L})}\leq C \|f\|_{{\mathbb L}^2(\mathbb R_+; {\rm BMO}_{\mathcal L})}.
$$

Indeed, for the case $4r_B\geq \rho(x_B)$, this estimate also holds by a similar argument presented in our article. However, for the remaining case $4r_B<\rho(x_B)$, the left-hand side of the above estimate involves 
$$
       \left |\nabla_x e^{-(t-s)\mathcal L}f(\cdot, s)(x)-\big(\nabla_x e^{-(t-s)\mathcal L}f(\cdot, s)\big)_B  \right|
$$
for $x\in B$, 
and the lack of regularity on heat kernels prevents us from handling it whenever $V\in {\rm RH}_{q}$ for $q> n/2$. 

In particular, if restricting $V\in {\rm RH}_{q}$ for $q\geq n$, then the regularity on heat kernels will be good enough (see \cite[Lemma 3.8]{DYZ}) to process this problem.
\end{remark}

One may continue to consider the global-in-time ${\rm CMO}_{\mathcal L}$-maximal regularity, where ${{\mathbb L}^2}   (\mathbb R_+;{\rm CMO}_{\mathcal L}(\mathbb R^n))$ can be defined in terms of limiting behavior conditions in a similar manner.   
Using the details arguments presented in this section, we obtain the following characterization.

\begin{corollary}\label{thm:global-mainB}
Suppose $V \in \mathrm{RH}_q$ for some $q>n / 2$. For every $f\in  {{\mathbb L}^2}(\mathbb R_+;{\rm CMO}_{\mathcal L}(\mathbb R^n))$,  the estimates \eqref{eqn:global-MaxReg-II} and \eqref{eqn:global-MaxReg-III} in Theorem \ref{thm:improved-mainA} also hold. Additionally,
$$	 \lim _{a \rightarrow 0}\sup _{B: r_{B} \leq a}  \,\mathtt {C}(f)_{B}
     =  \lim _{a \rightarrow \infty}\sup _{B: r_{B} \geq a}  \,\mathtt {C}(f)_{B}
     =   \lim _{a \rightarrow \infty}\sup _{B: B \subseteq \left(B(0, a)\right)^c}  \,\mathtt {C}(f)_{B}=0,
$$
where
$$
 \mathtt {C}(f)_{B}=\bigg( r_B^{-n} \int_{\mathbb R_+}\int_B
	 \left|\mathcal M_{\mathcal L}  (f)(x,t)-\big(\mathcal M_{\mathcal L } (f)(\cdot,t)\big)_B   \right|^2 dxdt\bigg)^{1/2}  .
$$
\end{corollary}

\vskip 1cm

 \noindent{\bf Acknowledgments.} Duong and Li are supported by ARC DP 220100285. Wu is  supported by  National Natural Science Foundation of China \# 12201002  and Anhui Natural Science Foundation of China \# 2208085QA03.

 \bigskip


\end{document}